\documentclass[a4paper]{amsart}
\usepackage{amssymb,amsmath,booktabs,color}
\usepackage{todonotes}
\usepackage{esint}

\newcommand{\ts}{\textstyle}
\newcommand{\trC}{\tr\lower2pt\hbox{$_\bC$}}
\newcommand{\qart}{1/4}
\newcommand{\7}{}  

\newtheorem{theorem}{Theorem}[section]

\newtheorem{corollary}[theorem]{Corollary}
\newtheorem{lemma}[theorem]{Lemma}
\newtheorem{proposition}[theorem]{Proposition}
\theoremstyle{definition}
\newtheorem{definition}[theorem]{Definition}
\newtheorem{example}[theorem]{Example}
\theoremstyle{remark}
\newtheorem{remark}[theorem]{Remark}

\newcommand{\cL}{{\mathcal L}}

\newcommand{\hook}{\lrcorner \,}
\newcommand{\re}{\mathrm{Re}}

\newcommand{\G}{\mathrm{G}}

\newcommand{\bR}{{\mathbb R}}

\newcommand{\bN}{{\mathbb N}}
\newcommand{\bC}{{\mathbb C}}
\newcommand{\tr}{\mathrm{tr}}

\newcommand{\SU}{{\mathrm{SU}}}
\newcommand{\diag}{{\mathrm{diag}}}

\newcommand{\mfg}{\mathfrak{g}}
\newcommand{\mfh}{\mathfrak{h}}
\newcommand{\mfb}{\mathfrak{b}}
\newcommand{\mfa}{\mathfrak{a}}

\newcommand{\mfk}{\mathfrak{k}}

\newcommand{\mfn}{\mathfrak{n}}

\newcommand{\ad}{\operatorname{ad}}
\newcommand{\Ann}{\operatorname{Ann}}

\newcommand{\End}{\operatorname{End}}

\newcommand{\spa}[1]{\operatorname{span}(#1)}

\newcommand{\GL}{\operatorname{GL}}
\newcommand{\SL}{\operatorname{SL}}

\numberwithin{equation}{section}

\begin{document}

\title[Closed $\G_2$-eigenforms and exact $\G_2$-structures]
{\Large Closed $\G_2$-eigenforms and exact $\G_2$-structures}
\author{Marco Freibert}
\address{Mathematisches Seminar\\
	Christian-Albrechts-Universit\"at zu Kiel\\
	Ludewig-Meyn-Strasse 4\\
	D-24098 Kiel\\
	Germany}
\email{freibert@math.uni-kiel.de}
\author{Simon Salamon}
\address{Mathematics Department \\
	King's College London\\
	Strand \\
	London \\
	WC2R 2LS \\
	United Kingdom}
\email{simon.salamon@kcl.ac.uk}

\date{}


\begin{abstract}
  A study is made of left-invariant $\G_2$-structures with an exact
  3-form on a Lie group $G$ whose Lie algebra $\mfg$ admits a codimension-one
  nilpotent ideal $\mfh$. It is shown that such a Lie group $G$ cannot
  admit a left-invariant closed $\G_2$-eigenform for the Laplacian
  and that any compact solvmanifold $\Gamma\backslash G$ arising from $G$ does not admit an (invariant) exact $\G_2$-structure. We also classify the seven-dimensional Lie algebras $\mfg$ with codimension-one ideal equal to the complex Heisenberg Lie algebra which admit exact $\G_2$-structures with or without special torsion. To achieve these goals, we first determine the six-dimensional nilpotent Lie algebras $\mfh$ admitting an exact $\SL(3,\bC)$-structure $\rho$ or a half-flat $\SU(3)$-structure $(\omega,\rho)$ with exact $\rho$, respectively.
\end{abstract}

\maketitle

\vskip-20pt

\section{Introduction}
The group $\G_2$ is one of the exceptional cases in Berger's celebrated list \cite{Be} of restricted holonomy groups of non-locally symmetric irreducible Riemannian manifolds and only occurs in dimension seven. For over $30$ years, it was unknown whether such manifolds exist at all until Bryant found local examples \cite{Br2}, Bryant and the second author found complete ones \cite{BrSa}, and Joyce \cite{J} constructed compact manifolds with $\G_2$ holonomy.

The construction of these examples relies on the fact that the metric is encoded in a certain type of three-form, which we shall refer to as a $\G_2$-structure. More exactly, a \emph{$G_2$-structure} on a seven-dimensional manifold $M$ is a three-form $\varphi\in \Omega^3 M$ on $M$ with pointwise stabilizer conjugate to $\G_2\subseteq \mathrm{SO}(7)\subseteq \GL(7,\bR)$. The form $\varphi$ induces a Riemannian metric $g_{\varphi}$, an orientation and a Hodge star operator $\star_{\varphi}$ on $M$. The holonomy group of $g_{\varphi}$ is contained in $\G_2$ if the structure is torsion-free, meaning that $\varphi$ is parallel for the Levi-Civita connection, which is the case if and only if $\varphi$ is closed and coclosed \cite{FG}.

$\G_2$-structures that are closed but \emph{not} coclosed constitute a basic intrinsic torsion class in the Fern\'{a}ndez-Gray classification, and play a natural role in the construction of compact manifolds with holonomy equal to $\G_2$. Joyce's examples were found by first constructing closed $\G_2$-structures on smooth manifolds with sufficiently small intrinsic torsion and then proving analytically that such closed $\G_2$-structures may be deformed to torsion-free ones.

Closed $\G_2$-structure are the initial values for the \emph{Laplacian flow} 
$\dot{\varphi}_t=\Delta_{\varphi_t}\varphi_t$
for one-parameter families of closed $\G_2$-structures $(\varphi_t)_{t\in I}$ introduced by Bryant in \cite{Br3}. The critical points of this flow are precisely the torsion-free $\G_2$-structures \cite{LW1}, and the hope is to use the Laplacian flow to deform a closed $\G_2$-structure (without any smallness assumption on the intrinsic torsion) to a torsion-free one for $t\rightarrow \infty$. 

Short-time existence and uniqueness of the Laplacian flow were established in \cite{BrXu}, and other foundational properties were proven in series of papers by Lotay and Wei \cite{LW1,LW2,LW3}. However, a lot is still unknown about long-time behaviour of the flow, and it is important to characterise finite-time singularities. One expects that, like for the Ricci flow, these singularities are modeled by self-similar solutions of the Laplacian flow.
The initial values $\varphi_0$ of these self-similar solutions are called \emph{Laplacian solitons}, and a special class of them is given by \emph{closed $\G_2$-eigenforms} characterised by
\begin{equation*}
\Delta_{\varphi_0}\varphi_0= \mu \varphi_0
\end{equation*}
for some $\mu\in \bR\setminus \{0\}$. Although this equation looks quite easy, no examples of these structures are known. Moreover, compact manifolds cannot admit a closed $\G_2$-eigenform \cite{LW1}.

Closed $\G_2$-eigenforms are also of interest from another point of view: they constitute a special class of so-called \emph{$\lambda$-quadratic closed $\G_2$-structures}, $\lambda\in \bR$, namely those with $\lambda=0$. In general, quadratic closed $\G_2$-structures are exactly the closed $\G_2$-structures for which the exterior derivative $d\tau$ of the associated torsion two-form $\tau$ depends quadratically on $\tau$. These structures have been studied by Ball \cite{Ba1,Ba2} and include many other interesting closed $\G_2$-structures. For example, the case $\lambda=\tfrac{1}{6}$ corresponds to so-called \emph{extremally Ricci-pinched (ERP)} closed $\G_2$-structures, and the case $\lambda=\tfrac{1}{2}$ is equivalent to the induced metric being Einstein.

By Lauret's work \cite{L}, homogeneous $\lambda$-quadratically closed $\G_2$-structures on homogeneous manifolds can only exist for $\lambda\in \{0,\tfrac{1}{6},\tfrac{1}{2}\}$. Homogeneous ERP closed $\G_2$-structures were classified in \cite{Ba2} using the classification of left-invariant such structures on Lie groups in \cite{LN2}. Moreover, \cite{FFM} shows that no solvable Lie group can admit a left-invariant closed Einstein $\G_2$-structure. Since the \emph{strong Alekseevsky conjecture} is true in dimension seven (i.e.\ any simply-connected homogeneous 7-Einstein manifold of negative scalar curvature is isometric to a left-invariant metric on a simply-connected solvable Lie group) \cite{AL}, there are no closed homogeneous Einstein $\G_2$-structures. So the homogeneous case is settled for $\lambda\in \{\tfrac{1}{6},\tfrac{1}{2}\}$ leaving open the case $\lambda=0$, where nothing at all is known.

We shall fill this gap as follows. Let $G$ be a 7-dimensional
  Lie group with Lie algebra $\mfg$. We prove that $G$ cannot admit a
  left-invariant closed $\G_2$-eigenform if $\mfg$ is \emph{almost
    nilpotent}, i.e.\ it admits a codimension-one nilpotent ideal. We
  are led to focus on ideals of two types, $\mfn_9$ and $\mfn_{28}$,
  with the former of step 4, and the latter of step 2 and isomorphic
  to the real Lie algebra underlying the complex Heisenberg group. It
  is striking that our non-existence proof is at the limit of, but just
  within, the realm of computations that can be checked by hand. This
  fact has enabled us to complement our conclusions with more positive
  ones relating to $\mfn_{28}$, mentioned below.

We are naturally led to the class of almost nilpotent Lie algebras by
the following facts:

Lauret and Nicolini \cite{LN1} showed that any Lie algebra $\mfg$ that possesses a closed $\G_2$-structure has a codimension-one unimodular ideal $\mfh$. Hence, it is quite natural to start with those for which $\mfh$ is nilpotent. 

Motivation is also provided by the results of Podest\`a and Raffero \cite{PR} on closed $\G_2$ structures on seven-manifolds with a transtive reductive group of symmetries.

Moreover, a closed $\G_2$-eigenform $\varphi$ is always (cohomologically) exact, and the existence problem of exact $\G_2$-structures on a restricted class of almost nilpotent Lie algebras has been studied in \cite{FFR}: there are no exact $\G_2$-structures on strongly unimodular Lie algebras $\mfg$ with $b_2(\mfg)=b_3(\mfg)=0$. The latter implies that $\mfg$ is almost nilpotent \cite{MaSw}, whereas `strongly unimodular' is a technical condition, necessary for the existence of a cocompact lattice in the associated simply-connected Lie group $G$.

We also answer negatively the existence problem for exact
$\G_2$-structures on strongly unimodular almost nilpotent Lie
algebras, and so on compact almost nilpotent (completely solvable)
solvmanifolds, thereby generalising the result of \cite{FFR}.
It is not known if there exists \emph{any} compact manifold
  with an exact $\G_2$-structure, though it \emph{is} known that (in
  contrast to other situations) nilmanifolds cannot serve as
  examples.

We do succeed in classifying all almost nilpotent Lie algebras admitting an exact $\G_2$-structure for which the codimension-one nilpotent ideal is isomorphic to $\mfn_{28}$. For such almost nilpotent Lie algebras, we also classify those that admit exact $\G_2$-structures with special torsion of positive or negative type, a notion introduced by Ball in \cite{Ba2}.

To prove our results, we split our almost nilpotent Lie algebra $\mfg$ as a vector space into $\mfg=\mfh\oplus \bR e_7$ with $\mfh$ being the codimension-one nilpotent ideal and $e_7\in \mfh^{\perp}$ of norm one. Then the equations determining a closed $\G_2$-eigenform or an exact $\G_2$-structure can be encoded into conditions on the induced $\SU(3)$-structure $(\omega,\rho)$ on $\mfh$. In particular, for an exact $\G_2$-structure, the $\SL(3,\bC)$-structure $\rho$ has to be exact and for a closed $\G_2$-eigenform, 
$(\omega,\rho)$ has to be half-flat with $\rho$ being the exterior derivative of a primitive $(1,1)$-form $\nu$. The extra equation $\nu\wedge\omega^2=0$ turns out to be of crucial importance in enabling us to rule out solutions to the eigenform equations.

We show that exactly five out of $34$ six-dimensional nilpotent Lie algebras admit an exact $\SL(3,\bC)$-structure and that exactly two of them admit a half-flat $\SU(3)$-structure $(\omega,\rho)$ with $\rho$ exact, namely $\mfn_9$ and $\mfn_{28}$. Both results have independent interest because special kinds of closed and exact $\SL(3,\bC)$-structures on six-dimensional nilpotent Lie algebras have been studied in \cite{FS}, and the six-dimensional nilpotent Lie algebras admitting a half-flat $\SU(3)$-structure $(\omega,\rho)$ with $d\omega=\rho$ were determined in \cite{FR}. Moreover, the result on exact $\SL(3,\bC)$-structures implies that if an almost nilpotent Lie algebra $\mfg$ admits an exact $\G_2$-structure, then the codimension-one nilpotent ideal $\mfh$ has to be one of the five Lie algebras. We provide examples of exact $\G_2$-structures on almost nilpotent Lie algebras with codimension-one nilpotent ideal $\mfh$ for all possible nilpotent Lie algebras $\mfh$ except when $\mfh$ equals the nilpotent Lie algebra called $\mfn_4$.

This leaves open the question to be studied in future work: Is there an almost nilpotent Lie algebra with codimension-one nilpotent ideal isomorphic to $\mfn_4$ that admits an exact $\G_2$-structure?
\clearpage
The paper is organised as follows.

In Section \ref{sec:preliminaries}, we summarise basic facts about $\SL(3,\bC)$-, $\SU(3)$- and $\G_2$-structures that are relevant to our investigation. In Section \ref{sec:reductionto6d}, we show how one can reduce the existence problem of a closed $\G_2$-eigenform or an exact $\G_2$-structure on a seven-dimensional Lie algebra $\mfg$ to the existence problem of $\SL(3,\bC)$- or $\SU(3)$-structures satisfying certain equations on a six-dimensional ideal $\mfh$ in $\mfg$. Next, in Section \ref{sec:resultsin6d}, we prove our results on exact $\SL(3,\bC)$-structures and on half-flat $\SU(3)$-structures $(\omega,\rho)$ with exact $\rho$. We use these results to prove in Section \ref{sec:exactG2} that no strongly unimodular almost nilpotent Lie algebra, and so also no compact almost nilpotent (completely solvable) solvmanifold, can admit an exact $\G_2$-structure. Finally, in Section \ref{sec:closedG2eigenforms}, we carried out a detailed analysis of the respective cases $\mfn_9$ and $\mfn_{28}$ in order to show that no almost nilpotent Lie algebra can admit a closed $\G_2$-eigenform. Moreover, we prove the mentioned classification results of almost nilpotent Lie algebras with a codimension ideal isomorphic to $\mfn_{28}$ admitting exact $\G_2$-structures.
\section{Preliminaries}\label{sec:preliminaries}
\subsection{$G$-structures in six and seven-dimensions}
In this subsection, we define three different types of $G$-structures in six and seven dimensions, and recall some of their basic properties. Proofs of the relevant facts and more information may be found, for example, in \cite{Br2,Br3,Hi}.

In all cases, the $G$-structure is defined by one or two differential forms which are pointwise isomorphic to one or two `model' forms on $\bR^n$, $n=6$ or $n=7$, whose $\GL(n,\bR)$-stabiliser is $G$. Here, \emph{pointwise isomorphic} means that for each $p\in M$ there is a vector space isomorphism $u:T_p M\rightarrow \bR^n$ that identifies the differential forms at the point $p\in M$ with the model forms on $\bR^n$.
\begin{definition}$\hphantom.$\\[-10pt]
\begin{itemize}
\item
An \emph{$\SL(3,\bC)$-structure} on an oriented six-dimensional manifold is a three-form $\rho\in \Omega^3 M$ which is pointwise isomorphic to
\begin{equation*}
\rho_0:=e^{135}-e^{146}-e^{236}-e^{245}\in \Lambda^3 (\bR^6)^*.
\end{equation*}
\item
An \emph{$\SU(3)$-structure} on a six-dimensional manifold is a pair $(\omega,\rho)$ of a two-form $\omega\in \Omega^2 M$ and a three-form $\rho\in \Omega^3 M$ which is pointwise isomorphic to $(\omega_0,\rho_0)$ with
\begin{equation*}
\omega_0:=e^{12}+e^{34}+e^{56}\in \Lambda^2 (\bR^6)^*.
\end{equation*}
\item
A \emph{$\G_2$-structure} on a seven-dimensional manifold $M$ is a three-form $\varphi\in \Omega^3 M$ which is pointwise isomorphic to
\begin{equation*}
\varphi_0:=\omega_0\wedge e^7+\rho_0\in \Lambda^3 (\bR^7)^*. 
\end{equation*}
\end{itemize}
In all three cases, if $u:T_p M\rightarrow \bR^n$ is one of the pointwise isomorphisms, then the basis $(u^{-1}(e_1),\ldots,u^{-1}(e_n))$ of $T_p M$ is called an \emph{adapted} basis for the $G$-structure in question. Sometimes, we will also call the dual basis of $(u^{-1}(e_1),\ldots,u^{-1}(e_n))$ an \emph{adapted} basis for the $G$-structure in question.
\end{definition}
Since $\SL(3,\bC)\subseteq \GL(3,\bC)$ and $\GL(3,\bC)$-structures are almost complex structures, an $\SL(3,\bC)$-structure $\rho$ has to induce an almost complex structure $J_{\rho}$. Explicitly, $J_{\rho}$ is obtained as follows:
\begin{definition}
Let $\rho$ be an $\SL(3,\bC)$-structure on an oriented six-dimensional manifold $M$. Then $\rho$ induces an almost complex structure $J=J_{\rho}$ on $M$ defined in $p\in M$ to be the unique endomorphism $J_p$ of $T_p M$ satisfying
\begin{equation*}
J_p f_{2i-1}=-f_{2i},\qquad J_p f_{2i}=f_{2i-1}.
\end{equation*}
for one, and so any, adapted oriented basis $(f_1,\ldots,f_6)$ of $T_p M$.

Moreover, set $\hat{\rho}:=J^*\rho\in \Omega^3 M$. Then
\begin{equation*}
\hat{\rho}_p=f^{246}-f^{235}-f^{145}-f^{136},
\end{equation*}
where $(f^1,\ldots,f^6)$ is the dual basis of the adapted basis $(f_1,\ldots,f_6)$ at $p\in M$. Furthermore, $\Psi:=\rho+i\hat{\rho}\in \Omega^3(M,\bC)$ is a non-zero $(3,0)$-form.
\end{definition}
We give an equivalent characterisation of an $\SL(3,\bC)$-structure and for this have to introduce a quartic invariant $\lambda$ of a three-form on a vector space:
\begin{definition}
Let $V$ be a six-dimensional vector space. Let $\kappa:\Lambda^5 V^*\rightarrow V\otimes \Lambda^6 V^*$ be the natural $\GL(V)$-equivariant isomorphism, i.e.\ $\kappa^{-1}(v\otimes \nu)=v\hook \nu$. Next, let $\rho\in \Lambda^3 V^*$ and define $K_{\rho}\in \End(V)\otimes \Lambda^6 V^*$ by
\begin{equation*}
K_{\rho}(v)=\kappa((v\hook \rho)\wedge \rho)
\end{equation*}
and finally set
\begin{equation*}
\lambda(\rho):=\tfrac{1}{6}\tr\left(K_{\rho}^2\right)\in (\Lambda^6 V^*)^{\otimes 2}.
\end{equation*}
\end{definition}
It makes sense to say that $\lambda(\rho)>0$, meaning that $\lambda(\rho)$ is the square of some element in $\Lambda^6 V^*$. Thus, one may also speak of $\lambda(\rho)<0$. Using this notation, Hitchin \cite{Hi} gives the following characterisation of an $\SL(3,\bC)$-structure:
\begin{lemma}\label{le:quarticinvariant}
Let $\rho\in \Omega^3 M$ be a three-form on an oriented six-dimensional manifold $M$. Then $\rho$ is an $\SL(3,\bC)$-structure if and only if $\lambda(\rho_p)<0$ for all $p\in M$.
\end{lemma}
We will also need the following technical statements in the sequel:
\begin{lemma}\label{le:SL3Clinearlydependent}
Let $\rho\in \Omega^3 M$ be an $\SL(3,\bC)$-structure on a seven-dimensional manifold. Let $p\in M$ and $v,w\in T_p M$ and set $J:=J_{\rho}$. Then:
\begin{enumerate}
	\item[(a)]
$\rho_p(v,w,\cdot)=0$ if and only if $v$, $w$ are $\bC$-linearly dependent.
 \item[(b)]
If $v\neq 0$, then the two-forms $\omega_1:=\rho_p(v,\cdot,\cdot)\in \Lambda^2 T_p^* M$ and $\omega_2:=\rho_p(J_pv,\cdot,\cdot)\in \Lambda^2 T_p^* M$ satisfy
\begin{equation*}
\ker(\omega_i)=\spa{v,J_pv},\qquad \omega_i\wedge \omega_j=\delta_{ij} \omega_1^2
\end{equation*}
for all $i,j=1,2$.
\end{enumerate}
\end{lemma}
\begin{proof}
\begin{enumerate}
	\item[(a)]
First, let $v,w$ are $\bC$-linearly dependent. Without loss of generality, we may assume that $w=c Jv$
for some $c\in \bR$. Since $\Psi$ is a $(3,0)$-form, we do have $\Psi(v,Jv,u)=-\Psi(Jv,J^2v,u)=\Psi(Jv,v,u)=-\Psi(v,Jv,u)$, i.e. $\Psi(v,Jv,u)=0$
for any $u\in T_p M$. As $\rho=\re(\Psi)$ this implies $\rho(v,w,\cdot)=c\rho(v,Jv,\cdot)=0$.

Next, assume that $v$ and $w$ are $\bC$-linearly independent. Then we may extend $v$ and $w$ to a $\bC$-basis of $(T_p M,J)$ by an element $u\in T_p M$. Then $v_{\bC}:=v-i Jv$, $w_{\bC}:=w-iJw$ and $u_{\bC}:=u-iJu$ form a basis of $(T_p M)^{1,0}$ and so $0\neq \Psi(v_{\bC},w_{\bC},u_{\bC})$ since $\Psi$ is a $(3,0)$-form. This property of $\Psi$ also shows
\begin{equation*}
\Psi(z-iJz,\cdot,\cdot)=\Psi(z,\cdot,\cdot)-i\Psi(Jz,\cdot,\cdot)=\Psi(z,\cdot,\cdot)-i^2 \Psi(z,\cdot,\cdot)=2\Psi(z,\cdot,\cdot)
\end{equation*}
for any $z\in T_p M$ and so $\Psi(v_{\bC},w_{\bC},u_{\bC})=8 \Psi(v,w,u)$. Consequently,
$\Psi(v,w,u)\neq 0$. Thus, $\rho(v,w,u)\neq 0$ or $\hat{\rho}(v,w,u)\neq 0$. In the latter case, we
do have
\begin{equation*}
\rho(v,w,Ju)+i\hat{\rho}(v,w,Ju)=\Psi(v,w,Ju)=i\Psi(v,w,u)=i^2\hat{\rho}(v,w,u)=-\hat{\rho}(v,w,u)\neq 0,
\end{equation*}
i.e.\ $\rho(v,w,Ju)\neq 0$.
\item[(b)]
Since all equations are invariant under non-zero rescalings, we may assume that $v$ has norm one. Now $\SU(3)$ acts transitively on the six-sphere $S^6$. Consequently, there is an adapted basis $(f_1,\ldots,f_6)$ of $\rho$ at $p\in M$ with $f_1=v$, and so $f_2=J_p f_1=J_p v$. But so
\begin{equation*}
\begin{split}
\omega_1&=v\hook \rho_p=f_1\hook \left(f^{135}-f^{146}-f^{236}-f^{245}\right)=f^{35}-f^{46},\\
\omega_2&=(J_pv)\hook \rho_p=f_2\hook \left(f^{135}-f^{146}-f^{236}-f^{245}\right)=-f^{36}-f^{45},\\
\end{split}
\end{equation*}
Then a straightforward computation shows that $\omega_1$ and $\omega_2$ have the desired properties.
\end{enumerate}
\end{proof}
Similarly, one knows that $\SU(3)\subseteq \mathrm{SO}(6)$ and so an $\SU(3)$-structure induces a Riemannian metric $g$ as follows:
\begin{definition}
Let $(\omega,\rho)$ be an $\SU(3)$-structure on a six-dimensional manifold $M$. Define $g=g_{(\omega,\rho)}$ to be the Riemannian metric on $M$ for which any adapted basis $(f_1,\ldots,f_6)$ at any point $p\in M$ is orthonormal.
Now $\omega^3$ is a volume form on $M$ and so $\omega$ induces an orientation on $M$. We get an induced almost complex structure $J_{\rho}$, which relates $g$ to $\omega$ by the equation
\begin{equation*}
g=\omega(J_{\rho}\cdot,\cdot).
\end{equation*}
Hence, $(g,J,\omega)$ is an almost Hermitian structure on $M$. Moreover, $\Psi=\rho+i\hat{\rho}$ is of constant length.
\end{definition}
Next, we turn to the special class of $\SU(3)$-structures defined in \cite{ChSa}:
\begin{definition}
An $\SU(3)$-structure $(\omega,\rho)$ on a six-dimensional manifold $M$ is called \emph{half-flat} if $d\omega^2=0$ and $d\rho=0$.
\end{definition}
Finally, we turn to $\G_2$-structures and use that $\G_2\subseteq \mathrm{SO}(7)\subseteq \GL^+(7,\bR)$:
\begin{definition}
  Let $\varphi\in \Omega^3 M$ be a $\G_2$-structure on a seven-dimensional manifold. Define $g=g_{\varphi}$ to be the Riemannian metric on $M$ for which any adapted basis $(f_1,\ldots,f_7)$ at any point $p\in M$ is orthonormal. Similarly, define an orientation on $M$ by requiring that any adapted basis $(f_1,\ldots,f_7)$ at any point $p\in M$ is oriented. We get an induced Hodge star operator $\star_{\varphi}$ and
\begin{equation*}
(\star_{\varphi}\varphi)_p=f^{1234}+f^{1256}+f^{3456}+f^{1367}+f^{1457}+f^{2357}-f^{2467}
\end{equation*}
for any $p\in M$ and any adapted basis $(f_1,\ldots,f_7)$ at $p$ with dual basis $(f^1,\ldots,f^7)$.
Moreover, $\star_{\varphi}\varphi$ is pointwise isomorphic to
\begin{equation*}
\star_{\varphi_0}\varphi_0=\tfrac{1}{2}\omega_0^2+e^7\wedge \hat{\rho}_0\in \Lambda^4 (\bR^7)^*.
\end{equation*}
\end{definition}
A \emph{$\G_2$-structure on a seven-dimensional vector space $V$} is simply a constant $\G_2$-structure on the manifold $V$, or said more directly, a three-form $\varphi\in \Lambda^3 V^*$ for which there exists an vector space isomorphism $u:V\rightarrow \bR^7$ with $u^*\varphi_0=\varphi$. So $\G_2$-structures $\varphi\in \Omega^3 M$ on seven-dimensional manifolds $M$ are those for which $\varphi_p$ is a $\G_2$-structure on the seven-dimensional vector space $T_p M$ for all $p\in M$.

Similarly, we define an \emph{$\SU(3)$-structure on a six-dimensional vector space $V$} and get the following result:
\begin{lemma}\label{le:SU3G2relation}
Let $\varphi\in \Lambda^3 V^*$ be a $\G_2$-structure on a seven-dimensional vector space $V$. 
Moreover, let $v\in V$ be of norm one with respect to $g_{\varphi}$ and let 
$W:=v^{\perp_{g_{\varphi}}}$. Then there is a unique $\SU(3)$-structure $(\omega,\rho)\in\Lambda^2 W^*\times \Lambda^3 W^*$ on $W$ such that
	\begin{equation*}
	\varphi=\omega\wedge \alpha+\rho,\qquad \star_{\varphi}\varphi=\tfrac{1}{2}\omega^2+\alpha\wedge \hat{\rho},
	\end{equation*}
with $\alpha\in V^*$ uniquely defined by $\alpha(W)=0$ and $\alpha(v)=1$, and $W^*$ identified with the annihilator of $v$.
\end{lemma}
\begin{proof}
$\G_2$ acts transitively on the unit sphere in $\bR^7$. Hence, we may assume that $\varphi$ has an adapted basis $(f_1,\ldots,f_7)$ with $v=f_7$ and so $\alpha=f^7$. Since $(f_1,\ldots,f_7)$ is an orthonormal basis of $V$, we have $W=\spa{f_1,\ldots,f_6}$ and the statements follow from the relations
$\varphi_0=\omega_0\wedge e^7+\rho_0$ and $\star_{\varphi_0}\varphi_0=\tfrac{1}{2}\omega_0^2+e^7\wedge \hat{\rho}_0$ between the model forms of a $\G_2$- and an $\SU(3)$-structure.
\end{proof}
Lemma \ref{le:SU3G2relation} yields the following non-degeneracy condition for $\G_2$-structures, whose proof we omit:
\begin{lemma}\label{le:G2nondegeneracy}
Let $\varphi\in \Lambda^3 V^*$ be a $\G_2$-structure on a seven-dimensional vector space $V$ and let $v,w\in V$ be linearly independent. Then $\varphi(v,w,\cdot)\neq 0$.
	\end{lemma}
Next, we consider some representation theory of $\G_2$. Consider the $\G_2$-representation $\Lambda^k(\bR^7)^*$ for $k\in\{0,\ldots,7\}$. We decompose this representation into its irreducible components, which are by now well-known. Since the Hodge star operator $\star_{\varphi_0}$ is an isomorphim of $\G_2$-representations between $\Lambda^k(\bR^7)^*$ and $\Lambda^{7-k}(\bR^7)^*$, it suffices to do this for $k\in \{0,\ldots,3\}$. 
Obviously, $\Lambda^0(\bR^7)^*$ is trivial, and $\Lambda^1(\bR^7)^*\cong \bR^7$ is also irreducible. For $k=2,3$, we have:
\begin{equation*}
\begin{split}
\Lambda^2 (\bR^7)^*&=\Lambda^2_7 \7\oplus \Lambda^2_{14} \7,\\
\Lambda^3 (\bR^7)^*&=\bR\varphi_0\oplus \Lambda^3_7 \7\oplus \Lambda^3_{27} \7
\end{split}
\end{equation*}
with
\begin{equation*}
\begin{split}
\Lambda^2_7 \7&=\left\{v\hook \varphi_0\left|v\in \bR^7\right.\right\},\quad \Lambda^2_{14} \7=\left\{\left.\nu\in \Lambda^2(\bR^7)^*\right| \nu\wedge \varphi_0=-\star_{\varphi_0}\nu\right\},\\
\Lambda^3_7 \7&=\left\{v\hook \star_{\varphi_0}\varphi_0\left|v\in \bR^7\right.\right\},\quad
\Lambda^3_{27} \7=\left\{\left.\gamma\in \Lambda^3(\bR^7)^*\right| \gamma\wedge \varphi_0=0,\, \gamma\wedge \star_{\varphi_0}\varphi_0=0\right\}.
\end{split}
\end{equation*}
The subscript denotes the dimension of the irreducible representation. For example, $\Lambda^2_{14} \7$ is isomorphic to the adjoint representation $\mfg_2$, as can be seen by using the metric $g_{\varphi_0}$ to identify a two-form $\nu\in \Lambda^2_{14} \7$ with an endomorphism of $\bR^7$.
The decompositions of $\Lambda^k (\bR^7)^*$ into irreducible $\G_2$-representations give rise to corresponding decompositions of $\Omega^k(M)$. In particular, $\Omega^2_{14} (M)=\left\{\nu\in \Omega^2 (M)\left| \nu\wedge \varphi=\star_{\varphi}\nu\right.\right\}$ is a $C^{\infty}(M)$-submodule of $\Omega^2(M)$.
\subsection{Closed $\G_2$-structures}
In this subsection, we consider the following situation.
\begin{definition}
A $\G_2$-structure $\varphi\in \Omega^3 M$ on a seven-dimensional manifold $M$ is called \emph{closed} if $d\varphi=0$.
If $\varphi$ is a closed $\G_2$-structure, then
\begin{equation*}
d\star_{\varphi}\varphi=\tau\wedge \varphi
\end{equation*}
for a unique two-form $\tau \in \Omega^2_{14}(M)$. The two-form $\tau$ is called the \emph{torsion two-form} of 
$\varphi$ and encodes the intrinsic torsion of $\varphi$.
\end{definition}
The notion of a closed $\G_2$-structure with \emph{special torsion of positive or negative type} was introduced by Ball in \cite{Ba2} in an attempt to study so-called \emph{quadratic closed $\G_2$-structures}, as explained in the next subsection.

To motivate the definition of closed $\G_2$-structure with special torsion, let $\varphi\in \Omega^3 M$ be a closed $\G_2$-structure with associated torsion two-form $\tau$. Then $\tau$ pointwise lies in the adjoint representation $\mfg_2$. The adjoint action of $\G_2$ on $\mfg_2$ has three different types of orbits, distinguished by the conjugacy classes of the $\G_2$-stabiliser at some point in the orbit:

There is the `generic' case, where the stabiliser is a maximal torus $T^2$ inside of $\G_2$.

Then there are two exceptional orbits, where the stabiliser is some copy of $\mathrm{U}(2)$, but the two copies $\mathrm{U}(2)^+$ and $\mathrm{U}(2)^-$ are not conjugate to each other. The first orbit $\G_2/\mathrm{U}(2)^+$ is the twistor space of the Wolf space $\G_2/SO(4)$, whereas $\G_2/\mathrm{U}(2)^-$ can be regarded as a twistor space of $S^6\cong\G_2/\SU(3)$ and as such is biholomorphic to a complex quadric \cite{Br1}.

We say that $\varphi$ has \emph{special torsion of positive type} or \emph{negative type}, respectively, if the pointwise stabiliser of $\tau$ is in each point conjugate to $\mathrm{U}(2)^+$ or $\mathrm{U}(2)^-$, respectively. As shown in \cite{Ba2}, these conditions can be characterised by properties of $\tau^3$:
\begin{definition}
Let $\varphi\in \Omega^3 M$ be a closed $\G_2$-structure on a seven-dimensional manifold $M$ with associated torsion two-form $\tau\in \Omega^2 M$.
Then $\varphi$ has \emph{special torsion of positive type} if $\tau^3=0$, and $\varphi$ has \emph{special torsion of negative type} if $\left|\tau^3\right|_{\varphi}^2=\tfrac{2}{3}\left|\tau\right|_{\varphi}^6$.
\end{definition}
\subsection{Closed $\G_2$-eigenforms for the Laplacian}
In this subsection, we discuss properties of closed $\G_2$-eigenforms (from now on, we shall omit the words `for the Laplacian'), and their relation to other
structures. We repeat the definition:
\begin{definition}
A $\G_2$-structure $\varphi$ on a seven-dimensional manifold $M$ is called a \emph{closed $\G_2$-eigenform} if $d\varphi=0$ and there exists some $\mu\in \bR\setminus \{0\}$ such that
\begin{equation*}
\Delta_{\varphi} \varphi=\mu \varphi.
\end{equation*}
\end{definition}
\begin{remark}$\hphantom.$\\[-10pt]
	\begin{itemize}
       \item
       Any closed $\G_2$-structure $\varphi$ (on a connected manifold) with $\Delta_{\varphi} \varphi= f \varphi$ for some $f\in C^{\infty}(M)$, $f\neq 0$, is a closed $\G_2$-eigenform. For, differentiation gives
       \begin{equation*}
       df\wedge \varphi=d(f\varphi)=d\Delta_{\varphi} \varphi=d^2 \delta_{\varphi}\varphi=0,
       \end{equation*}
       and so $df=0$ since wedging with $\varphi$ injects $\Omega^1 M$ into $\Omega^4 M$. Hence, $f$ is constant.
		\item 
	 A closed $\G_2$-eigenform $\varphi$ is an exact $\G_2$-structure since
		\begin{equation*}
			\varphi=\tfrac{1}{\mu} \Delta_{\varphi} \varphi=d\left(\tfrac{\delta_{\varphi} \varphi}{\mu}\right).
		\end{equation*}
		\item
		A closed $\G_2$-eigenform $\varphi$ is an example of a \emph{Laplacian soliton}, i.e.\ a soliton for the \emph{Laplacian flow (of closed $\G_2$-structures)} given by
		\begin{equation*}
			\dot{\varphi}_t=\Delta_{\varphi_t}\varphi_t.
		\end{equation*}
		More generally, a \emph{Laplacian soliton} $\varphi$ is a closed $\G_2$-structure satisfying
		\begin{equation*}
			\Delta_{\varphi} \varphi=\mu \varphi+\cL_X \varphi
		\end{equation*}
		for some $X\in \mathfrak{X}(M)$, $\mu\in \bR$.
		\item
		Lotay and Wei showed in \cite{LW1} that a compact seven-dimensional manifold cannot support any Laplacian soliton $\varphi$ with $X=0$ unless $\varphi$ is torsion-free. In particular, there do not exist any closed $\G_2$-eigenforms on compact manifolds.
		\item
		In \cite{LW1} it is also shown that for a closed $\G_2$-eigenform $\varphi\in \Omega^3 M$ with 
		$\Delta_{\varphi} \varphi=\mu \varphi$ we must have $\mu>0$.
	\end{itemize}
\end{remark}
Let $\varphi$ be a closed $\G_2$-eigenform on a seven-dimensional manifold $M$. By the last remark, we then have $\Delta_{\varphi} \varphi=\mu \varphi$ for some $\mu>0$.
Hence, by scaling $\varphi$ by an appropriate non-zero factor, we may and will assume for the rest of the article (unless stated otherwise) that $\mu=1$, i.e.\ that
\begin{equation}\label{eq:eigenform}
\Delta_{\varphi} \varphi=\varphi
\end{equation}
In the last subsection, we introduced the torsion $2$-form $\tau\in \Omega_{14}^2(M)$ uniquely defined by
\begin{equation*}
d\star_{\varphi}\varphi=\tau\wedge \varphi.
\end{equation*}
Since $\tau\in \Omega_{14}^2(M)$, we do have $\tau\wedge \varphi=-\star_{\varphi}\tau$ and so
\begin{equation*}
\tau=\star_{\varphi}\star_{\varphi} \tau=-\star_{\varphi} (\tau\wedge \varphi)=-\star_{\varphi} d\star_{\varphi}\varphi=\delta_{\varphi}\varphi,
\end{equation*}
which yields
\begin{equation}\label{eq:eigenformtorsion2form}
\Delta_{\varphi} \varphi=d\delta_{\varphi}\varphi=d\tau
\end{equation}
\begin{remark} 
Let $\varphi$ be a closed $\G_2$-eigenform. Then $d\tau=\Delta_{\varphi} \varphi=\mu \varphi$ for some $\mu>0$, which we do not assume to be equal to $1$ in this remark. Since wedging with
$\star_{\varphi}\varphi$ is pointwise $\G_2$-equivariant from $\Omega^2(M)$ to $\Omega^6 M$ and $\Omega^6 M$ is pointwise isomorphic to the $\G_2$-representation $\bR^7$, we must have $\tau\wedge \star_{\varphi}\varphi=0$. Thus,
\begin{equation*}\ts
7\mathrm{vol}_{\varphi}=\varphi\wedge \star_{\varphi}\varphi=\frac{1}{\mu} d\tau\wedge \star_{\varphi}\varphi=-\frac{1}{\mu} \tau\wedge d\star_{\varphi}\varphi=-\frac{1}{\mu} \tau^2\wedge \varphi=\frac{1}{\mu} \tau\wedge \star_{\varphi}\tau=\frac{\left|\tau\right|^2_{\varphi}}{\mu} \mathrm{vol}_{\varphi},
\end{equation*}
i.e.\ $\mu=\frac{1}{7}\left|\tau\right|^2_{\varphi}$, and so
\begin{equation*}\ts
d\tau=\frac{1}{7}\left|\tau\right|^2_{\varphi} \varphi.
\end{equation*}
In particular, $\left|\tau\right|_{\varphi}$ is constant. Moreover, closed $\G_2$-eigenforms are special kinds of so-called \emph{$\lambda$-quadratic closed $\G_2$-structures},
which are closed $\G_2$-structures $\varphi\in \Omega^3 M$ fulfilling
\begin{equation*}
d\tau= \tfrac{1}{7} \left|\tau\right|_{\varphi}^2 \varphi+ \lambda \left(\tfrac{1}{7} \left|\tau\right|_{\varphi}^2 \varphi+\star_{\varphi}(\tau\wedge \tau)\right)
\end{equation*}
for $\lambda\in \bR$, namely those for $\lambda=0$.

Note that $\lambda\big(\tfrac{1}{7} \left|\tau\right|_{\varphi} \varphi+\star_{\varphi}(\tau\wedge \tau)\big)$ lies in $\Omega^3_{27}(M)$, so the above decomposition can be seen as one of $d\tau\in \Omega^3 M$ into the three components
$\Omega^3_1(M):=C^{\infty}(M)\cdot \varphi$,  $\Omega^3_7(M)$ and $\Omega^3_{27}(M)$ of $\Omega^3 M$ with the $\Omega^3_7(M)$-component being zero.

More generally, for any closed $\G_2$-structure $\varphi$, the $\Omega^3_1(M)$-part of $d\tau$ equals
 $\frac{1}{7}\left|\tau\right|^2_{\varphi} \varphi$ and the $\Omega^3_7(M)$-part of $d\tau$ vanishes, i.e.\ we always have
 $d\tau=\tfrac{1}{7} \left|\tau\right|_{\varphi}^2 \varphi+\gamma$ for some $\gamma\in \Omega^3_{27}(M)$. 

One can show that $\lambda$-quadratic closed $\G_2$-structure are exactly those closed $\G_2$-structures for which $\gamma\in \Omega^3_{27}(M)$, and so the entire three-form $d\tau$, depends quadratically on $\tau$, explaining the naming of these structures.
\end{remark}
We restrict now to left-invariant $\G_2$-structures on seven-dimensional Lie groups $G$. These will from now on be identified with the corresponding structures on the associated seven-dimensional Lie algebra $\mfg$. 

As stated already, it is well known that a seven-dimensional nilpotent Lie algebra cannot admit an exact $\G_2$-structure, see e.g.\ \cite{CF}, and so we have to look for exact $\G_2$-structures on the more general class of solvable Lie algebras. 
We will give now a new proof of this fact, and in the process prove a slightly stronger result. Fist, we recall the following.
\begin{definition}
	Let $\mfk$ be a Lie algebra.
	
	The \emph{descending central series} $\mfk^0,\mfk^1,\ldots$ of $\mfk$ is defined by $\mfk^0:=\mfk$, $\mfk^1:=[\mfk,\mfk]$ and inductively by $\mfk^k:=[\mfk,\mfk^{k-1}]$ for all $k\in \bN$.
	
	The \emph{ascending central series} $\mfk_0,\mfk_1,\ldots$ of $\mfk$ is defined by $\mfk_0:=\{0\}$, $\mfk_1:=\mathfrak{z}(\mfk)$ and inductively by
	$\mfk_k:=\left\{X\in \mfk| [X,\mfk]\subseteq \mfk^{k-1}\right\}$ for all $k\in \bN$. 
	
	Note that $\mfk$ is nilpotent if and only if $\mfk^r=\{0\}$ for some $r\in \bN$ or, equivalently, if and only if $\mfk_s=\mfk$ for some $s\in \bN$ (and then $r=s$).
\end{definition}
This allows us to prove:
\begin{proposition}\label{pro:noexactG2nilpotentLA}
	Let $\mfg$ be a seven-dimensional Lie algebra. If $\mfg$ admits an exact $\G_2$-structure, then $\dim(\mathfrak{z}(\mfg))\leq 1$ and $\mfg_2=\mathfrak{z}(\mfg)$. In particular, if $\mfg$ is nilpotent, or, more generally, if $\mfg=\mfa\oplus \mfb$ is a direct sum of Lie algebras $\mfa$, $\mfb$ with $\mfb$ nilpotent and $\dim(\mfb)\geq 2$, then it cannot admit any exact $\G_2$-structure.
\end{proposition}
\begin{proof}
	Assume that $\varphi\in \Lambda^2 \mfg^*$ is an exact $\G_2$-structure on $\mfg$, i.e.\ $d\chi=\varphi$ for some $\chi\in \Lambda^2 \mfg^*$. Let $X,Y\in \mathfrak{z}(\mfg)$. Then
	\begin{equation*}
	0=d\chi(X,Y,Z)=\varphi(X,Y,Z)
	\end{equation*}
	for any $Z\in \mfg$. Hence, by Lemma \ref{le:G2nondegeneracy}, the vectors $X$ and $Y$ have to be linearly dependent. Thus, $\dim(\mathfrak{z}(\mfg))\leq 1$.
	
	If $\dim(\mathfrak{z}(\mfg))=0$, then trivially $\mfg_2=\mathfrak{z}(\mfg)$.
	So let us assume that $\dim(\mathfrak{z}(\mfg))=1$ and take $X\in \mathfrak{z}(\mfg)\setminus \{0\}$. If
	$\mfg_2\neq\mathfrak{z}(\mfg)$, then there exists $Y\in \mfg_2$ linearly independent of $X$ and we get
	$[Y,Z]\in \spa{X}$ and so
	\begin{equation*}
	0=-\chi([Y,Z],X)=d\chi(X,Y,Z)=\varphi(X,Y,Z)
	\end{equation*}
	for any $Z\in \mfg$, which contradicts Lemma \ref{le:G2nondegeneracy}. Thus, $\mfg_2=\mathfrak{z}(\mfg)$.
	
	So $\mfg$ certainly cannot be nilpotent nor can it be of the form $\mfg=\mfa\oplus \mfb$ with $\mfb$ being nilpotent and $\dim(\mfb)\geq 2$.
\end{proof}

\section{Reduction to six dimensions}\label{sec:reductionto6d}
In this section, we reduce the existence problem of closed $\G_2$-eigenforms or exact $\G_2$-structures on a seven-dimensional Lie algebra $\mfg$ to the existence of $\SU(3)$-structures of a certain type on a codimension-one ideal $\mfh$ satisfying specific equations. By \cite[Proposition 3.2]{LN1}, such an ideal $\mfh$ always exists:
\begin{proposition}\label{pro:codimoneideal}
	Let $\mfg$ be a seven-dimensional Lie algebra admitting a closed $\G_2$-structure $\varphi\in \Lambda^3 \mfg^*$. Then $\mfg$ admits a unimodular
	codimension-one ideal $\mfh$.
\end{proposition}
To obtain the reduction from seven to six dimensions, we need to recall what a derivation of a Lie algebra $\mfh$ is and how an endomorphism of $\mfh$ acts on $\Lambda^* \mfh^*$:
\begin{definition}
Let $\mfh$ be a Lie algebra, $f\in \End(\mfh)$ be an (vector space) endomorphism of $\mfh$ and $\alpha\in \Lambda^k \mfh^*$ be a $k$-form on $\mfh$. Then the $k$-form
$f.\alpha\in \Lambda^k \mfh^*$ is defined by
\begin{equation*}
(f.\alpha)(X_1,\ldots,X_k):=-\left(\alpha(f(X_1),X_2,\ldots,X_k)+\ldots+\alpha(X_1,\ldots,X_{k-1},f(X_k))\right).
\end{equation*}
A \emph{derivation of $\mfh$} is a (vector space) endomorphism $f\in \End(\mfh)$ of $\mfh$ such that $f([X,Y])=[f(X),Y]+[X,f(Y)]$ for all $X,Y\in \mfh$.
\end{definition}
\begin{remark}
Let $\mfh$ be a Lie algebra, $\alpha\in \Lambda^k \mfh^*$, $\beta\in \Lambda^l\mfh^*$ and $f,g\in \mathrm{End}(\mfh)$. Then:
\begin{itemize}
	\item
	Due to the global minus sign in the definition of $f.\alpha$, we have $[f,g].\alpha=f.(g.\alpha)-g.(f.\alpha)$, i.e.\ $ \End(\mfh)\ni f\mapsto f.\in \End(\Lambda^k \mfh^*)$ is a representation of the Lie algebra $\End(\mfh)$ on $\Lambda^k \mfh^*$.
	\item
	Moreover, we have
	\begin{equation*}
	f.(\alpha\wedge \beta)=f.\alpha\wedge \beta+\alpha\wedge f.\beta
	\end{equation*}
	and so $f.(\alpha\wedge \alpha)=2 \alpha\wedge f.\alpha$ if $k$ is even.
	\item
	$f$ is a derivation if and only if $f.d\gamma=d(f.\gamma)$ for all one-forms $\gamma\in \mfh^*$ on $\mfh$ and then the same formula holds for forms of arbitrary degree on $\mfh$. Moreover, the vector space $\mathrm{Der}(\mfh)$ of all derivations of $\mfh$ is a subalgebra of the Lie algebra $\mathrm{End}(\mfh)$ of all (vector space) endomorphisms of $\mfh$ and it is the Lie algebra of the Lie group $\mathrm{Aut}(\mfh)$ of all Lie algebra automorphisms of $\mfh$, i.e. of
	\begin{equation*}
	\mathrm{Aut}(\mfh):=\left\{F\in \GL(\mfh)|F([X,Y])=[F(x),F(Y)]\textrm{ for all } X,Y\in \mfh\right\}\subseteq \GL(\mfh).
	\end{equation*}
\end{itemize}
\end{remark}
Let us begin with the reduction to six dimensions:

For this, let $\varphi$ be a closed $\G_2$-eigenform on a seven-dimensional Lie algebra $\mfg$ and $\mfh$ be a codimension-one ideal. Choose $e_7\in \mfg$ of norm one in the orthogonal complement $\mfh^{\perp_{g_{\varphi}}}$ of $\mfh$ in $\mfh$. Split now $\mfg=\mfh\oplus \spa{e_7}$ and, similarly,
$\mfg^*=\mfh^*\oplus \spa{e^7}$, where $e^7$ is the unique element in the annihilator of $\mfh$ with $e^7(e_7)=1$ and $\mfh^*$ is identified with the annihilator of $e_7$. Set $f:=\ad(e_7)|_{\mfh}$ and note that $f$ is a derivation of $\mfh$. Then
\begin{equation*}
d\alpha=d_{\mfh}\alpha+e^7\wedge f.\alpha,\quad d(\alpha\wedge e^7)=d_{\mfh} \alpha\wedge e^7,\quad d e^7=0
\end{equation*}
for any $\alpha\in \Lambda^k \mfh^*$, where $d_{\mfh}$ is the differential of $\mfh$. Next, decompose $\varphi$ according to the splitting,
i.e.\ write
\begin{equation}\label{eq:varphidecomposition}
\varphi=\omega\wedge e^7+\rho
\end{equation}
with $\omega\in \Lambda^2 \mfh^*$, $\rho\in \Lambda^3 \mfh^*$. Then $(\omega,\rho)$ is an $\SU(3)$-structure on $\mfh$ by Lemma \ref{le:SU3G2relation} and one has
\begin{equation*}
\star_{\varphi}\varphi=\tfrac12\omega^2+e^7\wedge \hat{\rho}.
\end{equation*}
We do the same for the torsion-two form $\tau\in \Lambda^2_{14}\mfg^*$ of $\varphi$, i.e.\ we write
\begin{equation}\label{eq:taudecomposition}
\tau=\nu+\alpha\wedge e^7
\end{equation}
with $\nu\in \Lambda^2 \mfh^*$ and $\alpha\in \mfh^*$. Now the $\G_2$-representation $\Lambda^2_{14}\mfg^*$ splits as $\SU(3)$-representations into $\Lambda^2_{14}\mfg^*=\mfh^*\wedge e^7\oplus [\Lambda^{1,1}_0 \mfh^*]$
as $\Lambda^2_{14} \mfg^*$ is the adjoint representation of $\G_2$ and $[\Lambda^{1,1}_0 \mfh^*]$ is the adjoint representation of $\SU(3)$. Thus, $\nu\in [\Lambda^{1,1}_0 \mfh^*]$
and so $\nu\wedge \rho=0$. Consequently, 
\begin{equation*}
\begin{split}\ts
 \frac{1}{2}\, d_{\mfh}(\omega^2) + e^7\wedge (\omega\wedge f.\omega-d_{\mfh} \hat{\rho})=d \star_{\varphi}\varphi=\tau\wedge \varphi&=(\nu+\alpha\wedge e^7)\wedge (\omega\wedge e^7+\rho)\\
&=e^7\wedge (\omega\wedge \nu-\alpha\wedge \rho),
\end{split}
\end{equation*}
i.e.
\begin{equation*}
d_{\mfh}(\omega^2)=0,\qquad \omega\wedge f.\omega-d_{\mfh} \hat{\rho}=\omega\wedge \nu-\alpha\wedge \rho.
\end{equation*}
Moreover,
\begin{equation*}
d_{\mfh} \nu+e^7\wedge (f.\nu-d_{\mfh} \alpha)=d\tau=\varphi=\omega\wedge e^7+\rho,
\end{equation*}
i.e.
\begin{equation*}
d_{\mfh} \nu=\rho, \qquad f.\nu-d_{\mfh} \alpha=\omega,
\end{equation*}
Hence, $d_{\mfh}\rho=0$ and $d_{\mfh}(\omega^2)=0$ and so the $\SU(3)$-structure $(\omega,\rho)$ on $\mfh$ is half-flat with exact $\rho$. Moreover, we necessarily have $\alpha=0$. For this, note that
$\tau\in \Omega^2_{14} M=\{\beta\in \Omega^2 M|\star_{\varphi}\beta=-\beta\wedge \varphi\}$ and so
\begin{equation*}
\tau=\star_{\varphi}^2 \tau=-\star_{\varphi} (\tau\wedge \varphi)=-\star_{\varphi}\left(e^7\wedge(\omega \wedge \nu-\alpha\wedge \rho)\right)\in \Lambda^2 \mfh^*
\end{equation*}
since $e^7$ is perpendicular to $\mfh^*$ by assumption. Consequently, $\tau=\nu\in \Lambda^3 \mfh^*$ and $\alpha=0$.

Summarizing, we have arrived at
\begin{theorem}\label{th:eigenform}
Let $\mfg$ be a seven-dimensional Lie algebra, $\varphi\in \Lambda^3 \mfg^*$ be a $\G_2$-structure on $\mfg$, $\mfh$ be a codimension-one unimodular ideal in $\mfg$, $e_7\in \mfh^{\perp_{g_{\varphi}}}$ of norm one, $e^7\in \Ann(\mfh)$ with $e^7(e_7)=1$ and $f:=\ad(e_7)|_{\mfh}$. Write
$\varphi=\omega\wedge e^7+\rho$ with $(\omega,\rho)\in \Lambda^2 \mfh^*\times \Lambda^3 \mfh^*$. Then $\varphi$ is a closed $\G_2$-eigenform with $\Delta_{\varphi} \varphi=\varphi$ if and only if $(\omega,\rho)$ is half-flat and there exists a primitive $(1,1)$-form $\nu\in [\Lambda^{1,1}_0 \mfh^*]$ on $\mfh$ such that $\rho=d_{\mfh}\nu$ and
\begin{align}
f.\nu&=\omega, \label{eq:1} \\
\omega\wedge f.\omega-d_{\mfh} \hat{\rho}&=\omega\wedge \nu \label{eq:2}.
\end{align}
\end{theorem}
If we only look for exact $\G_2$-structures $\varphi\in \Lambda^3 \mfg^*$, the same calculations as above show:
\begin{theorem}\label{th:exactG2structure}
Let $\mfg$ be a seven-dimensional Lie algebra, $\varphi\in \Lambda^3 \mfg^*$ be a $\G_2$-structure on $\mfg$, $\mfh$ be a codimension-one unimodular ideal in $\mfg$, $e_7\in \mfh^{\perp_{g_{\varphi}}}$ of norm one, $e^7\in \Ann(\mfh)$ with $e^7(e_7)=1$ and $f:=\ad(e_7)|_{\mfh}$. Write
$\varphi=\omega\wedge e^7+\rho$ with $(\omega,\rho)\in \Lambda^2 \mfh^*\times \Lambda^3 \mfh^*$. Then $\varphi$ is an exact $\G_2$-structure if and only if there exists a two-form $\nu\in \Lambda^2 \mfh^*$ on $\mfh$ and a one-form $\alpha\in \mfh^*$ with $\rho=d_{\mfh}\nu$ and 
\begin{equation}\label{eq:3}
f.\nu-d_{\mfh} \alpha=\omega.
\end{equation}
\end{theorem}
\section{Results in dimension six}\label{sec:resultsin6d}
From now on, we restrict to ourselves to a special class of Lie algebras:
\begin{definition}
	A Lie algebra $\mfg$ is called \emph{almost nilpotent} if it admits a codimension-one nilpotent ideal $\mfh$.
Note that then $\mfg\cong \mfh\rtimes_f \bR$ for a derivation $f\in \mathrm{Der}(\mfh)$, where $\mfh\rtimes_f \bR$ denotes the semi-direct product of $\bR$ with $\mfh$ and Lie algebra representation $\rho:\bR\rightarrow \mathrm{Der}(\mfh)$ of $\bR$ on $\mfh$ given by $\rho(t)=t f$ for all $t\in \bR$.
\end{definition}
In order to investigate the existence of exact $\G_2$-structures and closed $\G_2$-eigenforms on seven-dimensional almost nilpotent Lie algebras $\mfg$, we first have to determine which six-dimensional nilpotent Lie algebras admit exact $\SL(3,\bC)$-structures or half-flat $\SU(3)$-structures $(\omega,\rho)$ for which there exists a primitive $(1,1)$-form $\nu$ with $\rho=d\nu$.
\subsection{Exact $\mathrm{SL}(3,\bC)$-structures on nilpotent Lie algebras}
We start by determining the six-dimensional nilpotent Lie algebras $\mfh$ admitting an exact $\SL(3,\bC)$-structure $\rho\in  \Lambda^3\mfh^*$. To this aim,
we rephrase the condition of being exact in the following way:
\begin{proposition}\label{pro:exactrho}
Let $\mfh$ be a six-dimensional nilpotent Lie algebra. Then $\mfh$ admits an exact $\SL(3,\bC)$-structure $\rho\in \Lambda^3 \mfh^*$ if and only if there exist linear independent one-forms $\alpha_1,\alpha_2\in \mfh^*$ and two-forms $\omega_1,\omega_2\in \Lambda^2 \mfh^*$ with $\omega_i\wedge \omega_j=\delta_{ij} \omega_1^2$ for $i,j\in \{1,2\}$ such that $\ker(\omega_1)=\ker(\omega_2)$ is a complement of $\ker(\alpha_1)\cap \ker(\alpha_2)$ in $\mfh$ and such that either
\begin{itemize}
\item[(a)]
$\dim(\mathfrak{z}(\mfh))=1$, $\dim(\mfh_2)=2$, $\ker(\omega_1)=\ker(\omega_2)=\mfh_2$ and there exists a closed non-zero one-form $\gamma\in \mfh^*\setminus \{0\}$ with $\gamma(\mfh_2)=\{0\}$ such that
\begin{equation*}
d\alpha_1=\omega_1,\qquad d\alpha_2=\omega_2+\gamma\wedge \alpha_1
\end{equation*}
\item[(b)]
or $\dim(\mathfrak{z}(\mfh))=2$, $\ker(\omega_1)=\ker(\omega_2)=\mathfrak{z}(\mfh)$ and
\begin{equation*}
d\alpha_1=\omega_1,\qquad d\alpha_2=\omega_2.
\end{equation*}
\end{itemize}
In the first case, $\mfh_2$ is $J$-invariant and in the second case $\mathfrak{z}(\mfh)$ for the almost complex structure $J$ induced by $\rho$.
\end{proposition}
\begin{proof} {\it The forward implication.}
Assume that $\mfh$ admits an exact $\SL(3,\bC)$-structure $\rho$, i.e.\ $\rho=d\nu$ for some $\nu\in \Lambda^2 \mfh^*$. Let $X,Y\in \mathfrak{z}(\mfh)$. Then
\begin{equation*}
\rho(X,Y,Z)=d\nu(X,Y,Z)=0
\end{equation*}
for any $Z\in \mfh$, which implies that $X$ and $Y$ are $\bC$-linearly dependent by Lemma \ref{le:SL3Clinearlydependent} (a). So $\dim(\mathfrak{z}(\mfh))\in \{1,2\}$.

If $\dim(\mathfrak{z}(\mfh))=2$, $\mathfrak{z}(\mfh)$ is $J$-invariant, so we may choose a basis $X,JX$ of $\mathfrak{z}(\mfh)$.
By Lemma \ref{le:SL3Clinearlydependent} (b), $\omega_1:=\rho(X,\cdot,\cdot)$ and $\omega_2:=\rho(JX,\cdot,\cdot)$ have two-dimensional common kernel $\mathfrak{z}(\mfh)$ and fulfill
$\omega_i\wedge \omega_j=\delta_{ij} \omega_1^2$ for $i,j\in \{1,2\}$.

Moreover, setting $\alpha_1:=-\nu(X,\cdot)\in \mfh^*$ and $\alpha_2:=-\nu(JX,\cdot)\in \mfh^*$, we have
\begin{equation*}
d\alpha_1(Y,Z)=-\alpha_1([Y,Z])=-\nu([Y,Z],X)=d\nu(X,Y,Z)=\rho(X,Y,Z)=\omega_1(Y,Z)
\end{equation*}
for all $Y,Z\in \mfh$, i.e.\ $d\alpha_1=\omega_1$. In the same way, one obtains $d\alpha_2=\omega_2$, which then also shows that $\alpha_1$
and $\alpha_2$ are linearly independent. Now choose $Y\in \mfh_2$ linearly independent of $X$ and $JX$. By Lemma \ref{le:SL3Clinearlydependent} (a), there exists $Z\in \mfh$
with $\rho(X,Y,Z)\neq 0$ and so
\begin{equation*}
0\neq \rho(X,Y,Z)=\nu(X,[Y,Z]).
\end{equation*}
Since $[Y,Z]\in \mathfrak{z}(\mfh)=\spa{X,JX}$, this shows $\alpha_1(JX)=-\alpha_2(X)=\nu(X,JX)\neq 0$. Thus $\ker(\alpha_1)\cap\ker(\alpha_2)$ is
complementary to $\ker(\omega_1)=\ker(\omega_2)=\mathfrak{z}(\mfh)$.

Next, consider the case $\dim(\mathfrak{z}(\mfh))=1$ and choose $X\in \mathfrak{z}(\mfh)$ and $Y\in \mfh_2$ linearly independent. Then we have
\begin{equation*}
\rho(X,Y,Z)=d\nu(X,Y,Z)=\nu(X,[Y,Z])=0
\end{equation*}
for any $Z\in \mfh$, i.e. $X$ and $Y$ are $\bC$-linearly dependent by Lemma \ref{le:SL3Clinearlydependent} (a). Hence $\dim(\mfh_2)=2$ and $\mfh_2$ is $J$-invariant.

Choose a basis $X,JX$ of $\mfh_2$ such that $X\in \mathfrak{z}(\mfh)$ and set again $\omega_1:=\rho(X,\cdot,\cdot)$, $\omega_2:=\rho(JX,\cdot,\cdot)$,
$\alpha_1:=-\nu(X,\cdot)$ and $\alpha_2:=-\nu(JX,\cdot)$. As in the case $\dim(\mathfrak{z}(\mfh))=2$, we get $\ker(\omega_1)=\ker(\omega_2)=\spa{X,JX}=\mfh_2$, $\omega_i\wedge \omega_j=\delta_{ij} \omega_1^2$ for $i,j=1,2$ and $d\alpha_1=\omega_1$.

Next, let $Y\in \mfh_3$ linearly independent of $X$
and $JX$. By Lemma \ref{le:SL3Clinearlydependent} (a), we again have some $Z\in \mfh$ with $0\neq \rho(X,Y,Z)=\nu(X,[Y,Z])$ and from $[Y,Z]\in \mfh_2=\spa{X,JX}$ we get again that
$\alpha_1(JX)=-\alpha_2(X)=\nu(X,JX)\neq 0$, i.e.\ that $\ker(\alpha_1)\cap\ker(\alpha_2)$ is complementary to $\ker(\omega_1)=\ker(\omega_2)=\mathfrak{z}(\mfh)$.
So we finally have to prove the equation for $d\alpha_2$ in this case. Thereto, let $\gamma\in \mfh^*\setminus \{0\}$ be the one-form uniquely
defined by $[JX,Y]=-\gamma(Y)X$ for all $Z\in \mfh$. Obviously, $\gamma(X)=\gamma(JX)=0$, i.e.\ $\gamma(\mfh_2)=\{0\}$. Moreover, $d\gamma=0$ as
\begin{equation*}
\begin{split}
d\gamma(Z,W)X=-\gamma([Z,W])X=[JX,[Z,W]]&=[Z,[W,JX]]+[W,[JX,Z]]\\
&=-\gamma(W)[Z,X]+\gamma(Z)[W,X]\\&=0
\end{split}
\end{equation*}
for all $Z,W\in \mfh$. Furthermore,
\begin{equation*}
\begin{split}
d\alpha_2(Y,Z)=-\alpha_2([Y,Z])&=-\nu([Y,Z],JX)\\
&=d\nu(JX,Y,Z)+\nu([JX,Y],Z)-\nu([JX,Z],Y)\\
&=\rho(JX,Y,Z)-\nu(\gamma(Y)X,Z)+\nu(\gamma(Z)X,Y)\\
&=\omega_2(Y,Z)+\gamma(Y)\alpha_1(Z)-\gamma(Z)\alpha_1(Y)\\
&=\omega_2(Y,Z)+(\gamma\wedge \alpha_1)(Y,Z)
\end{split}
\end{equation*}
as claimed.

\noindent{\it The backwards implication.}
Assume that there exist linear independent one-forms $\alpha_1,\alpha_2\in \mfh^*$ and two-forms $\omega_1,\omega_2\in \Lambda^2 \mfh^*$ as in the statement. Note that $\dim(\ker(\alpha_1)\cap \ker(\alpha_2))=4$ since $\alpha_1,\alpha_2$ are linearly independent. Consequently, $\dim(\ker(\omega_1))=\dim(\ker(\omega_2))=2$ and so $\omega_1,\omega_2$ are non-degenerate two-forms
on $V:=ker(\alpha_1)\cap \ker(\alpha_2)$ satisfying $\omega_i\wedge \omega_j=\delta_{ij}\omega_1^2$
for $i,j=1,2$ and $\omega_1^2\neq 0$. Then it is well-known that there is a basis
$(v_1,\ldots,v_4)$ such with respect to the dual basis $(v^1,\ldots,v^4)$ we have
\begin{equation*}
\omega_1=v^{12}+v^{34},\qquad \omega_2=v^{13}-v^{24},
\end{equation*}
cf., e.g., the proof of Lemma 2.2 in \cite{FY}. Consider $(v^1,\ldots,v^4)$ as one-forms on $\mfh$ by identifying $V^*$ with the annihilator of $\ker(\omega_1)=\ker(\omega_2)$ and set
\begin{equation*}
\rho:=\alpha_2\wedge \omega_1-\alpha_1\wedge \omega_2=-\alpha^1\wedge v^{13}+\alpha^1\wedge v^{24}+\alpha^2\wedge v^{12}+\alpha^2\wedge v^{34}
\end{equation*}
$\rho$ is an $\SL(3,\bC)$-structure on $\mfh^*$ since an adapted basis is given by
$(\alpha_1,\alpha_2,v^3,v^2,v^1,-v^4)$. Moreover, $\rho$ is exact since
$\nu:=\alpha^{12}\in \Lambda^2 \mfh^*$ satisfies $d\nu=d\alpha^1\wedge \alpha^2-\alpha^1\wedge d\alpha^2=\alpha_2\wedge \omega_1-\alpha_1\wedge \omega_2=\rho$ in both cases.
\end{proof}
There are $34$ (isomorphism classes of) real six-dimensional nilpotent Lie algebras. Of these,
exactly those five admits an exact $\SL(3,\bC)$-structures which are listed in Table \ref{table:1}. The notation for these Lie algebras is obtained by numbering the $34$ six-dimensional nilpotent Lie algebras
from $\mfn_1$ to $\mfn_{34}$ in the order in which they occur in Table A.1 in \cite{Sa}.
\begin{corollary}\label{co:exactrho}
Let $\mfh$ be a six-dimensional nilpotent Lie algebra. Then $\mfh$ admits an exact $\SL(3,\bC)$-structure if and only if $\mfh$ is one of the five Lie algebras listed in Table \ref{table:1}.
\end{corollary}
\begin{table}[h]
	\centering
	\begin{tabular}{@{\vrule width 0pt height 2.5ex depth
				1.5ex}*{4}{l}@{}}
		\toprule
		$\mfg$ & differentials \\
		\midrule
		$\mfn_1$ & $(0,0,12,13,14+23,34-25)$  \\
		$\mfn_4$ & $(0,0,12,13,14+23,24+15)$  \\
		$\mfn_9$ & $(0,0,0,12,14-23,15+34)$  \\
		$\mfn_{18}$ &  $(0,0,0,12,13+42,14+23)$  \\
		$\mfn_{28}$ & $(0,0,0,0,13+42,14+23)$  \\
		\bottomrule
	\end{tabular}
	\bigskip
	\caption{Six-dimensional nilpotent Lie algebras admitting an exact $\SL(3,\bC)$-structure}
	\label{table:1}
\end{table}

\begin{proof}
By Proposition \ref{pro:exactrho}, we either have $\dim(\mathfrak{z}(\mfh))=2$ or $\dim(\mathfrak{z}(\mfh))=1$ and $\dim(\mfh_2)=2$.

Let us first assume that $\dim(\mathfrak{z}(\mfh))=2$. By Proposition \ref{pro:exactrho}, there are closed two-forms
$\omega_1,\omega_2\in \Lambda^2 \mfh^*$ with common kernel $\ker(\omega_1)=\ker(\omega_2)=\mathfrak{z}(\mfh)$ such that $\omega_i\wedge \omega_j=\delta_{ij} \omega_1^2$. Moreover, there are linearly independent one-forms $\alpha_1,\alpha_2\in \mfh^*$ such that $d\alpha_i=\omega_i$ for $i=1,2$ and such that $\ker(\alpha_1)\cap \ker(\alpha_2)$ is complementary to $\mathfrak{z}(\mfh)$. Now $\mfh/\mathfrak{z}(\mfh)$ is a four-dimensional nilpotent Lie algebra and $\omega_1,\omega_2$ descend to closed two-forms on $\mfh/\mathfrak{z}(\mfh)$, again called $\omega_1,\omega_2$. 
It is well known, see e.g.\ \cite{Ov}, that there are exactly three four-dimensional nilpotent Lie algebras,
namely $(0,0,0,0)$, $(0,0,12,0)$ and $(0,0,12,13)$. One easily checks that only $(0,0,0,0)$ and $(0,0,12,0)$ admit closed two-forms $\omega_1,\omega_2$ with $\omega_i\wedge \omega_j=\delta_{ij} \omega_1^2$. 

If $\mfh/\mathfrak{z}(\mfh)\cong (0,0,0,0)$, then one may choose (cf. the proof of Proposition \ref{pro:exactrho}) a basis $e^1,\ldots,e^4$ of the dual space of $(0,0,0,0)$ such that $\omega_1=e^{13}-e^{24}$ and $\omega_2=e^{14}+e^{23}$. We may extend this basis to a basis $e^1,\ldots,e^6$ of $\mfh^*$ by $e^5:=\alpha^1$ and $e^6:=\alpha^2$
and so $\mfh\cong (0,0,0,0, 13+42,14+23)=\mfn_{28}$.

If $\mfh/\mathfrak{z}(\mfh)\cong (0,0,12,0)$, then $\omega_1$ is a symplectic form on $(0,0,12,0)$ and so symplectomorphic to
$e^{14}+e^{23}$ by \cite{Ov}, i.e.\ we may assume that $\omega_1=e^{14}+e^{23}$. Then, since $\omega_2$ is closed, $\omega_1\wedge \omega_2=0$ and $\omega_2^2=\omega_1^2$, one checks that $\omega_2=a (e^{14}-e^{23})+ b_1 e^{13}-b_2 e^{24}+ c e^{12}$ for certain $a,b_1,b_2,c\in \bR$ with $b_1b_2-a^2=1$. But so
\begin{equation*}
f:=\begin{pmatrix} 
1 & \tfrac{a}{b_1} &  0 & 0 \\
0 & 1  & 0 & 0 \\
0 & 0 & 1 & -\tfrac{a}{b_1} \\
-\tfrac{c}{b_2} & -\tfrac{ a c}{b_1 b_2} & 0 & 1
\end{pmatrix}
\end{equation*}
 is an automorphism of
$(0,0,12,0)$ with $f^*\omega_1=\omega_1$ and $f^*\omega_2= b_1 e^{13}-\frac{1}{b_1} e^{24}$. Next,
$$g:=\diag\big(b_1^{-\tfrac{2}{3}},b_1^{\tfrac{1}{3}},b_1^{-\tfrac{1}{3}},b_1^{\tfrac{2}{3}}\big)$$
fulfills $g^*f^*\omega_1=\omega_1$ and $g^*f^*\omega_2=e^{13}-e^{24}$. Extending $e^1,\ldots,e^4$ to a basis $e^1,\ldots,e^6$ by setting
$e^5:=\alpha_1$, $e^6:=\alpha_2$, we do get $\mfh\cong (0,0,12,0,14+23,13+42)\cong (0,0,0,12,13+42,14+23)=\mfn_{18}$, where the latter isomorphism $F$ is, e.g., given by the one with by $F(e_1)=-e_2$, $F(e_2)=e_1$, $F(e_3)=e_4$, $F(e_4)=-e_3$, $F(e_5)=e_6$ and $F(e_6)=e_5$.

Next, let $\dim(\mathfrak{z}(\mfh))=1$ and $\dim(\mfh_2)=2$. By Proposition \ref{pro:exactrho}, there are two-forms
$\omega_1,\omega_2\in \Lambda^2 \mfh^*$ with common kernel $\ker(\omega_1)=\ker(\omega_2)=\mfh_2$ such that $\omega_i\wedge \omega_j=\delta_{ij} \omega_1^2$. Moreover, there are linearly independent one-forms $\alpha_1,\alpha_2\in \mfh^*$ and $\gamma\in \mfh^*\setminus \{0\}$ closed with $\gamma(\mfh_2)=\{0\}$ such that $d\alpha_1=\omega_1$, $d\alpha_2=\omega_2+\gamma\wedge \alpha_1$ and such that $\ker(\alpha_1)\cap \ker(\alpha_2)$ is complementary to $\mathfrak{z}(\mfh)$. Note that then $\omega_1$ is closed and
\begin{equation*}
d\omega_2=\gamma\wedge \omega_1.
\end{equation*}
Hence, $\mfh_2\hook d\omega_2=0$ and so $\omega_1,\omega_2,\gamma$ descend to forms on $\mfa:=\mfh/\mfh_2$ with $d\omega_1=0$ and $d\omega_2=\gamma\wedge \omega_1$. Since $d\omega_2\neq 0$ on $\mfa$, $\mfa$ cannot be Abelian and we must either have
$\mfa\cong (0,0,12,0)$ or $\mfa\cong (0,0,12,13)$.

Let us first assume that $\mfa\cong (0,0,12,0)$. By the results in \cite{Ov}, all symplectic forms on $\mfa$ are symplectomorphic to each other. Hence, we may assume that $\omega_1= e^{13}-e^{24}$. Then $\omega_2= a_1 e^{12}+a_2 e^{34}+ b (e^{13}+ e^{24})+ c_1 e^{14}+ c_2 e^{23}$ for certain $a_1,a_2,b,c_1,c_2\in \bR$ with $a_1 a_2+c_1 c_2-b^2=1$. We must have $a_2\neq 0$ as otherwise $d\omega_2=0$, a contradiction. But so the automorphism 
\begin{equation*}
\begin{pmatrix} 
-\tfrac{1}{a_2} & 0 & 0 & 0 \\
0 & a_2^2 & 0 & 0 \\
\tfrac{c_1}{a_2^2} & - b a_2 & -a_2 & 0 \\
-\tfrac{b}{a_2^2} & c_2 a_2 & 0 & \tfrac{1}{a_2^2}
\end{pmatrix}
\end{equation*}
of $(0,0,12,0)$ is well-defined. That automorphism fixes $\omega_1$ and transforms $\omega_2$ into $-e^{12}-e^{34}$. Hence, we may assume that $\omega_1=e^{13}-e^{24}$ and that $\omega_2=-e^{12}-e^{34}$. Then $d\omega_2=-e^{124}=e^1\wedge \omega_1$, i.e.\ $\gamma=e^1$. Thus, extending $e^1,\ldots,e^4$ to a basis $e^1,\ldots,e^6$ of $\mfh^*$ by $e^5:=\alpha_1$ and
$e^6:=\alpha_2+e^3$, we have $\mfh=(0,0,12,0,13-24,15-34)\cong (0,0,0,12,14-23,15+34)=\mfn_9$, where the latter isomorphism fixes $e^i$
for $i\notin \{3,4\}$ and interchanges $e^3$ and $e^4$.

Next, let us consider the case $\mfa\cong (0,0,12,13)$. By \cite{Ov}, all symplectic two-forms on $\mfa$ are symplectomorphic. Thus, we may assume that $\omega_1=e^{14}+e^{23}$. Then $\omega_2=a_1 e^{12}+a_2 e^{34}+ b_1 e^{13}-b_2 e^{24}+ c (e^{14}-e^{23})$ for certain $a_1,a_2,b_1,b_2,c\in \bR$ with $a_1 a_2+ b_1 b_2-c^2=1$.

Let us first assume that $a_2\neq 0$. Then
\begin{equation*}
\begin{pmatrix} 
1 & 0 & 0 & 0 \\
-\tfrac{b_2}{a_2} & 1 & 0 & 0 \\
-\tfrac{a_2 c+ b_2^2}{a_2^2} & \tfrac{b_2}{a_2} & 1 & 0 \\
\tfrac{a_2 b_1+b_2 c}{a_2^2} & -\tfrac{c}{a_2}& \tfrac{b_2}{a_2} & 1
\end{pmatrix}
\end{equation*}
is an automorphism of $(0,0,12,13)$ which fixes $\omega_1$ and maps $\omega_2$ to $\frac{1}{a_2}e^{12}+a_2 e^{34}$, i.e.\ may assume that $\omega=\frac{1}{a_2}e^{12}+a_2 e^{34}$. Then $d\omega_2=a_2 e^{124}=-a_2 e^2\wedge \omega_1$, i.e.\ $\gamma=-a_2e^2$.
Hence, extending $e^1,\ldots,e^4$ to a basis $e^1,\ldots,e^6$ of $\mfh^*$ by $e^5:=\alpha_1$ and $e^6:=\frac{1}{a_2}(\alpha_2-\frac{1}{a_2}e^3)$, we have $\mfh=(0,0,12,13,14+23,34-25)=\mfn_1$.

Finally, we consider the case $a_2=0$. Then $b_1 b_2-c^2=1$ and so, in particular, $b_2\neq 0$. Thus,
\begin{equation*}
\begin{pmatrix}
1 & 0 & 0 & 0 \\
\tfrac{c}{b_2} & 1 & 0 & 0 \\
0 & -\tfrac{c}{b_2} & 1 & 0 \\
\tfrac{-a_1 b_2^2+c}{b_2^3}& \tfrac{c^2}{b_2^2} & -\tfrac{c}{b_2} & 1
\end{pmatrix}
\end{equation*}
is a well-defined automorphism of $(0,0,12,13)$ which fixes $\omega_1$ and maps $\omega_2$ to $\tfrac{1}{b_2} e^{13}-b_2 e^{24}$. Hence, we may assume that $\omega_2=\tfrac{1}{b_2} e^{13}-b_2 e^{24}$ and then $\gamma=-b_2 e^1$
as $d\omega_2=-b_2 e^{123}= -b_2 e^1\wedge \omega_1$. Thus, extending $e^1,\ldots,e^4$ to a basis $e^1,\ldots,e^6$ of $\mfh^*$ by $e^5:=\alpha_1$ and $e^6:=-\frac{1}{b_2}\left(\alpha_2-\frac{1}{b_2}e^4\right)$, we have $\mfh=(0,0,12,13,14+23,24+15)=\mfn_4$.

Conversely, the existence of forms as in Proposition \ref{pro:exactrho} follows from the discussion above on any of the Lie algebras $\mfn_1,\mfn_4,\mfn_9,\mfn_{18}$ and $\mfn_{28}$.
\end{proof}

\subsection{Half-flat $\SU(3)$-structures $(\omega,\rho)$ with exact $\rho$}
Here, we determine the six-dimensional nilpotent Lie algebras which admit a half-flat $\SU(3)$-structure $(\omega,\rho)$ for which $\rho=d\nu$ for a primitive $(1,1)$-form $\nu$.
In fact, we will determine all nilpotent Lie algebras which admit a half-flat $\SU(3)$-structure $(\omega,\rho)$ with exact $\rho$ and show that these are the same for which $\rho=d\nu$ with a primitive $(1,1)$-form $\nu$.

For this, note that by Corollary \ref{co:exactrho}, only $\mfn_1$, $\mfn_4$, $\mfn_9$, $\mfn_{18}$, $\mfn_{28}$ may admit a half-flat $\SU(3)$-structure $(\omega,\rho)$ with exact $\rho$. Now Conti determined the six-dimensional nilpotent Lie algebras admitting a half-flat $\SU(3)$-structures in \cite{C} and his results reduce the possible cases to $\mfn_4$, $\mfn_9$, $\mfn_{28}$. We will show that $\mfn_4$ cannot admit a half-flat $\SU(3)$-structure $(\omega,\rho)$ with exact $\rho$ and for the proof we use the following obstruction by Schulte-Hengesbach and the first author \cite{FrSH} adapted to our setting. Note that this obstruction is a refinement of one used by Conti in \cite{C}:
\begin{lemma}\label{le:obstructionhalflat}
  Let $\mfh$ be a six-dimensional Lie algebra and $\nu
  \in \Lambda^6\mfh^*\setminus \{0\}$. If there is a non-zero one-form $\alpha \in\mfh^*$ satisfying
  \begin{equation}
    \label{tildeJ:1}
    \alpha \wedge \tilde J^*_\tau \alpha \wedge \sigma = 0
  \end{equation} 
  for all exact three-forms $\tau \in \Lambda^3 \mfh^*$ and all closed four-forms $\sigma \in \Lambda^4 \mfh^*$, where $\tilde
  J_\tau^*\alpha$ is defined for $X \in \mfh^*$ by
  \begin{eqnarray}
    \label{tildeJ:2}
    \tilde J_{\tau}^* \alpha (X) \,\nu &=& 
    \alpha \wedge (X \hook \tau) \wedge \tau,
  \end{eqnarray}
  then $\mfg$ does not admit a half-flat $\SU(3)$-structure $(\omega,\rho)\in\Lambda^2 \mfh^*\times \Lambda^3 \mfh^*$ with exact $\rho$.
\end{lemma}
This allows us now to prove:
\begin{theorem}\label{th:halfflatexactrho}
Let $\mfh$ be a six-dimensional nilpotent Lie algebra. Then $\mfh$ admits a half-flat $\SU(3)$-structure $(\omega,\rho)\in \Lambda^2 \mfh^*\times\Lambda^3\mfh^*$ with exact $\rho$ if and only if $\mfh$ is isomorphic to $\mfn_9$ or $\mfn_{28}$. In these cases, $\mfh$ also admits a half-flat $\SU(3)$-structure $(\tilde{\omega},\tilde{\rho})\in \Lambda^2 \mfh^*\times\Lambda^3\mfh^*$  with $\tilde{\rho}=d\nu$ for some primitive $(1,1)$-form $\nu\in [\Lambda^{1,1}_0 \mfh^*]$.
\end{theorem}
\begin{proof}
As explained above, by the results of \cite{C} and Corollary \ref{co:exactrho}, only $\mfn_4$, $\mfn_9$ or $\mfn_{28}$ may admit a half-flat $(\omega,\rho)\in \Lambda^2 \mfh^*\times\Lambda^3\mfh^*$ with exact $\rho$ 

Now a direct computation, efficiently carried out with a computer algebra system like MAPLE, shows that one may the obstruction in Lemma \ref{le:obstructionhalflat}
with, e.g.\ $\alpha=e^1$ or $\alpha=e^2$, to exclude the existence of a half-flat $\SU(3)$-structure $(\omega,\rho)\in \Lambda^2 \mfh^*\times\Lambda^3\mfh^*$ with exact $\rho$ on $\mfn_4$.

For the other two cases, we provide a half-flat $\SU(3)$-structure $(\tilde{\omega},\tilde{\rho})\in \Lambda^2 \mfh^*\times\Lambda^3\mfh^*$ and some $\nu\in [\Lambda^{1,1}_0 \mfh^*]$ with $\tilde{\rho}=d\nu$:
\begin{itemize}
\item $\mfn_9$: Here, we may take the $\SU(3)$-structure defined by the adapted basis $(e^1,e^3,e^2,e^4,e^5,-e^6)$, i.e.\
\begin{equation*}
\omega=e^{13}+e^{24}-e^{56},\qquad \rho=e^{125}+e^{146}-e^{236}-e^{345}.
\end{equation*}
Then one checks that $d(\omega^2)=0$. Moreover, set $\nu:=e^{13}+\tfrac{1}{2}e^{26}+\tfrac{1}{2}e^{45}+e^{56}$. Then $d\nu=\rho$ and $\rho$ is a $(1,1)$-form.
Since $\nu\wedge \omega^2=0$, $\nu$ is primitive as well, i.e.\ $\nu\in [\Lambda^{1,1}_0 \mfh^*]$.
\item $\mfn_{28}$: Take the $\SU(3)$-structure defined by the adapted basis $(e^1,e^2,e^3,e^4,e^6,e^5)$, i.e.\
\begin{equation*}
\omega=e^{12}+e^{34}-e^{56},\qquad \rho=e^{136}-e^{145}-e^{235}-e^{246}.
\end{equation*}
Then $d(\omega^2)=0$. Setting now $\nu:=e^{12}+e^{56}$, we get $d\nu=\rho$ and that $\rho$ is a $(1,1)$-form.
Again $\nu\wedge \omega^2=0$ and so $\nu$ is primitive, i.e.\ $\nu\in [\Lambda^{1,1}_0 \mfh^*]$.
\end{itemize}
\end{proof}
\begin{remark}
Fino and Raffero determined in \cite{FR} all six-dimensional nilpotent Lie algebras admitting a so-called \emph{coupled} half-flat $\SU(3)$-structure, i.e.\ a half-flat $\SU(3)$-structure $(\omega,\rho)\in \Lambda^2 \mfh^*\times \Lambda^3\mfh^*$ with $d\omega=\rho$. Interestingly, the six-dimensional nilpotent Lie algebras admitting a coupled half-flat $\SU(3)$-structure are also $\mfn_9$ and $\mfn_{28}$.

Our proof of Theorem \ref{th:halfflatexactrho} is independent of the coupled approach, and in some sense more direct. Fino and Raffero compute with a computer algebra system for all $24$ six-dimensional nilpotent Lie algebras admitting a half-flat $\SU(3)$-structure the most general exact three-form $\rho$ and check if the quartic invariant $\lambda$ of $\rho$ can be negative. This way they obtain that, of the six-dimensional nilpotent Lie algebras admitting a half-flat $\SU(3)$-structure, those which admit a maybe non half-flat $\SU(3)$-structure with exact three-form part are precisely $\mfn_4$, $\mfn_9$ and $\mfn_{28}$. Then they show by different methods that $\mfn_4$ cannot admit a coupled $\SU(3)$-structure.
\end{remark}
\section{Exact $\G_2$-structures on compact almost nilpotent solvmanifolds}\label{sec:exactG2}
Here, we prove that a compact almost nilpotent solvmanifold cannot admit an invariant exact $\G_2$-structure.

For this, note first that Corollary \ref{co:exactrho} implies the following:
\begin{corollary}\label{co:exactG2}
	Let $\mfg$ be a seven-dimensional almost nilpotent Lie algebra with codimension-one nilpotent ideal $\mfh$. If $\mfg$ admits an exact $\G_2$-structure, then $\mfh$ is isomorphic
	to $\mfn_1$, $\mfn_4$, $\mfn_9$, $\mfn_{18}$ or $\mfn_{28}$.
\end{corollary}
We show now that four of the five cases of a codimension-one nilpotent ideal, namely $\mfn_1$, $\mfn_9$, $\mfn_{18}$ and $\mfn_{28}$, may occur in Corollary \ref{co:exactG2} leaving open if there is an almost nilpotent Lie algebra with codimension one ideal $\mfn_4$ which admits an exact $\G_2$-structure.

For this, note that Theorem \ref{th:exactG2n28} below even classifies all the almost nilpotent Lie algebra with codimension-one nilpotent ideal isomorphic to $\mfn_{28}$ which admit an exact $\G_2$-structure. For $\mfh\in \{\mfn_1, \mfn_9 \mfn_{18}\}$, we provide now one example of an exact $\G_2$-structure on a seven-dimensional almost nilpotent Lie algebra with codimension-one nilpotent ideal $\mfh$:
\begin{example}
	For the six-dimensional nilpotent Lie algebras $\mfh\in \{\mfn_1,\mfn_9,\mfn_{18}\}$, we give an $\SU(3)$-structure $(\omega,\rho)\in \Lambda^2 \mfh^*\times \Lambda^3 \mfh^*$, a two-form $\nu\in\Lambda^2 \mfh^*$, a one-form $\alpha\in \mfh^*$ and a derivation $f\in \mathrm{Der}(\mfh)$ such that $\rho=d\nu$ and such that \eqref{eq:3} is valid.
	\begin{itemize}
		\item $\mfn_1$: Take the $\SU(3)$-structure $(\omega,\rho)\in \Lambda^2 \mfn_1^*\times \Lambda^3 \mfn_1^*$ defined by the adapted basis
		$\big(-e^1+ae^5,e^3+ae^6,$ $e^2,e^4,e^5,e^6\big)$ with $a:=\tfrac{3+\sqrt{5}}{2}$, i.e.\
		\begin{equation*}
		\omega=-e^{13}-a\, e^{16}+e^{24}-a\, e^{35}+\left(1+a^2\right)e^{56},\qquad \rho=-e^{125}+e^{146}+e^{236}-e^{345}.
		\end{equation*}
		Setting 
		\begin{equation*}\ts
		\nu:=\left(2-\frac{2}{3}\,a\right) e^{15}-\frac{1}{2}e^{16}+\left(2-\frac{2}{3}\,a\right) e^{24}-\frac{1}{2} e^{35}+e^{56},
		\end{equation*}
		one gets $d\nu=\rho$. Moreover,
		\begin{equation*}
		f:=
		\left(\begin{smallmatrix}
		-\frac{a}{4} & 0 & 0 & 0 & 0 & 0 \\
		0 & -\frac{a}{2} & 0 & 0 & 0 & 0 \\
		0 & 0 &  -\frac{a^2+1}{4} & 0 & 0 & 0 \\
		0 & 0 & 0 & -a & 0 & 0  \\
		0 & 0 & 0 & 0 & -\frac{5}{4}\,a  & 0  \\
		-1 & 0 & 0 & 0 & 0 & -\frac{7}{4}\,a \\
		\end{smallmatrix}\right).
		\end{equation*}
		is a derivation of $\mfn_1$ and one computes $f.\nu=\omega+e^{13}$.
		Hence, choosing $\alpha:=e^4$, we have $d\alpha=e^{13}$ and so $f.\nu-d\alpha=\omega$, i.e.\ \eqref{eq:3} is fulfilled.
		\item $\mfn_9$:
		In this case, we choose the $\SU(3)$-structure $(\omega,\rho)\in \Lambda^2 \mfn_9^*\times \Lambda^3 \mfn_9^*$ defined by the adapted basis $(e^1,e^3,e^2,e^4,e^5,-e^6)$, i.e.	
		\begin{equation*}
		\omega=e^{13}+e^{24}-e^{56},\qquad \rho=e^{125}+e^{146}-e^{236}-e^{345}.
		\end{equation*}
		Setting
		\begin{equation*}\ts
		\nu:=-\frac{2}{3} e^{13}+e^{24}+\frac{1}{2} e^{26}+\frac{1}{2} e^{45}+e^{56},
		\end{equation*}
		one obtains $d\nu=\rho$. Moreover,
		\begin{equation*}\ts
		f:=\diag\left(\frac{1}{2},-\frac{3}{4},1,-\frac{1}{4},\frac{1}{4},\frac{3}{4}\right)
		\end{equation*}
		is a derivation of $\mfn_9$ and $f.\nu=e^{13}+e^{24}-e^{56}=\omega$. Thus, for $\alpha:=0$,
		\eqref{eq:3} is satisfied.
		\item $\mfn_{18}$: Here, we look at the $\SU(3)$-structure $(\omega,\rho)\in \Lambda^2 \mfn_{18}^*\times \Lambda^3 \mfn_{18}^*$ defined by the adapted basis $(e^1,e^2,e^3,e^4,e^6,e^5)$, i.e.
		\begin{equation*}
		\omega=e^{12}+e^{34}-e^{56},\qquad \rho=e^{136}-e^{145}-e^{235}-e^{246}.
		\end{equation*}
		Taking
		\begin{equation*}\ts
		\nu:=\frac{3}{2}e^{16}-\frac{3}{2}e^{34}+e^{56},
		\end{equation*}
		we get $d\nu=\rho$. Now one checks that
		\begin{equation*}
		f:=
		\left(\begin{smallmatrix}
		\frac{1}{6} & 0 & 0 & 0 & 0 & 0 \\
		0 & \frac{1}{6} & 0 & 0 & 0 & 0 \\
		0 & 0 & \frac{1}{3} & 0 & 0 & 0 \\
		0 & 0 & 0 & \frac{1}{3} & 0 & 0  \\
		-1 & 0 & 0 & 0 & \frac{1}{2} & 0  \\
		0 & 0 & 0 & 0 & 0 & \frac{1}{2}  \\
		\end{smallmatrix}\right).
		\end{equation*}
		is a derivation of $\mfn_{18}$ and that $f.\nu=e^{34}-e^{56}$. Thus, for $\alpha:=-e^4$, we have $d\alpha=-e^{12}$ and so $f.\nu-d\alpha=e^{12}+e^{34}-e^{56}=\omega$,
		i.e.\ \eqref{eq:3} is valid for our choices.
	\end{itemize}
\end{example}
Next, we look at compact almost nilpotent solvmanifolds, i.e. manifolds of the form $\Gamma\backslash G$, where $G$ is a simply-connected almost nilpotent Lie group and $\Gamma$ a cocompact lattice in $G$. A necessary condition for the existence of such a lattice is that the associated Lie algebra $\mfg$ is \emph{strongly unimodular}, cf.\ \cite{G}:
\begin{definition}
Let $\mfg$ be a solvable Lie algebra $\mfg$, $\mfn$ be its nilradical and $\mfn^0,\mfn^1,\ldots$ be the descending central series of $\mfn$. One checks that $\ad_X$ preserves $\mfn^i$ for all $X\in \mfg$ and all $i\in \bN$. $\mfg$ is called \emph{strongly unimodular} if $\tr(\ad_X|_{\mfn^i/\mfn^{i+1}})=0$ for all $i\in \bN$ and all $X\in \mfg$.
\end{definition}
\begin{remark}
Since the commutator ideal $[\mfg,\mfg]$ of a solvable Lie algera $\mfg$ is nilpotent, the nilradical $\mfn$ contains the commutator ideal $[\mfg,\mfg]$. Hence, if $\mfg$ is strongly unimodular, one has $\tr(\ad_X)=0$ for all $X\in \mfg$, i.e.\ $\mfg$ is unimodular.
\end{remark}
\begin{theorem}\label{th:exactG2stronglyunimodular}
Let $\mfg$ be a seven-dimensional strongly unimodular almost nilpotent Lie algebra. Then $\mfg$ does not admit an exact $\G_2$-structure.
\end{theorem}
\begin{proof}
Assume the contrary. By Corollary \ref{co:exactG2} the Lie algebra $\mfg$ then admits a codimension-one nilpotent ideal $\mfh$ which is isomorphic to $\mfn_1$, $\mfn_4$, $\mfn_9$, $\mfn_{18}$ or $\mfn_{28}$. Moreover, $\mfh$ is the nilradical as the entire Lie algebra cannot be nilpotent according to Proposition \ref{pro:noexactG2nilpotentLA}. Furthermore, we have an induced $\SU(3)$-structure $(\omega,\rho)\in \Lambda^2 \mfh^*\times \Lambda^3 \mfh^*$ on $\mfh$ with exact $\rho$, i.e.\ there is some $\nu\in \Lambda^2 \mfh^*$ with $d\nu=\rho$ which has to fulfill
	\begin{equation*}
	f.\nu-d_{\mfh}\alpha=\omega
	\end{equation*}
	for some one-form $\alpha\in \mfh^*$ by Theorem \ref{th:exactG2structure}. Now we know that in the cases
	$\mfn_1$, $\mfn_4$ and $\mfn_9$, we have $\dim(\mathfrak{z}(\mfh))=1$ and $\dim(\mfh_2)=2$ with $\mfh_2$ being $J$-invariant by Proposition \ref{pro:exactrho} for the almost complex structure $J$ induced by $\rho$. Moreover, in all theses cases, one checks that $\mfh_2$ is
	the sum of quotient spaces of the form $\mfh^i/\mfh^{i+1}$, i.e.\ the trace of each $\ad_X$, $X\in \mfg$, has to be trace-free on $\mfh_2$. In these cases, we set $\mfa:=\mfh_2$.
	
	In the cases $\mfn_{18}$ and $\mfn_{28}$, we have $\dim(\mathfrak{z}(\mfh))=2$ and $\mathfrak{z}(\mfh)$ is $J$-invariant by Proposition \ref{pro:exactrho}. Moreover, $\mathfrak{z}(\mfh)$ equals in both cases the last non-zero $\mfh^i$, so is of the form $\mfh^i/\mfh^{i+1}$. Hence, each $\ad_X$, $X\in \mfg$, has to be trace-free when restricted to $\mathfrak{z}(\mfh)$. Here, we set $\mfa:=\mathfrak{z}(\mfh)$.
	
	Now coming back to genereal case, choose some $0\neq X\in \mathfrak{z}(\mfh)\subseteq \mfa$. Then we get
	\begin{equation*}
	0=-\tr(f|_{\mfa})\, \nu(X,JX) = (f.\nu-d_{\mfh}\alpha)(X,JX)=\omega(X,JX)=-\left\|X\right\|^2\neq 0,
	\end{equation*}
	since $f$ has to preserve $\mfa=\spa{X,JX}$. This yields the desired contradiction and so $\mfg$ cannot admit an exact $\G_2$-structure.
\end{proof}
In general, if $G$ is a simply-connected solvable Lie group which admits a cocompact lattice $\Gamma$, then any left-invariant differential form $\beta$ induces a differential form $\tilde{\beta}$ on the compact quotient $\Gamma\backslash G$. We then call $\tilde{\beta}$ \emph{invariant}. By \cite[Theorem 3.2.10]{OTr}, the assignment $\beta\mapsto \tilde{\beta}$ induces an injection $H^*(\mfg)\rightarrow H_{dR}^*(\Gamma\backslash G)$.

Hence, Theorem \ref{th:exactG2stronglyunimodular} implies that no compact almost nilpotent solvmanifold can admit an \emph{invariant} exact $\G_2$-structure. If $G$ is \emph{completely solvable}, i.e.\ $\ad_X$ has only real eigenvalues for all $X\in \mfg$, then $H^*(\mfg)\rightarrow H_{dR}^*(\Gamma\backslash G)$ is an isomorphism by \cite{H} and so one may skip the word `invariant' in the statement:
\begin{corollary}\label{co:exactG2solvmanifolds}
Let $M=\Gamma\backslash G$ be an almost nilpotent solvmanifold, i.e.\ $G$ is a simply-connected almost nilpotent Lie group and $\Gamma$ a cocompact lattice in $G$. 
Then $M$ does not admit an invariant exact $\G_2$-structure. If $G$ is completely solvable, then $M$ does not admit any exact $\G_2$-structure at all.
\end{corollary}

\section{Closed $\G_2$-eigenforms on almost nilpotent Lie algebras}\label{sec:closedG2eigenforms}
In this section, we establish:
\begin{theorem}\label{th:noeigenforms}
Let $\mfg$ be a seven-dimensional almost nilpotent Lie algebra. Then
$\mfg$ does not admit a closed $\G_2$-eigenform.
\end{theorem}
To start the proof, note that by Theorem \ref{th:eigenform} and Theorem \ref{th:halfflatexactrho}, the codimension-one nilpotent ideal $\mfh$ of an almost nilpotent Lie algebra admitting a closed $\G_2$-eigenform has to be isomorphic to $\mfn_9$ or to $\mfn_{28}$.

In Subsection \ref{subsec:n9}, we will show in Theorem \ref{th:noeigenformsn9} that no almost nilpotent Lie algebra with codimension-one nilpotent ideal isomorphic to $\mfn_9$ can admit a closed $\G_2$-eigenform and in Subsection \ref{subsec:n28}, we will show in Theorem \ref{th:noeigenformsn28} that no almost nilpotent Lie algebra with codimension-one nilpotent ideal isomorphic to $\mfn_{28}$ can admit a closed $\G_2$-eigenform. This work completes the proof of Theorem \ref{th:noeigenforms}.

In Subsection \ref{subsec:n28}, we also give a classification of all almost nilpotent Lie algebras with codimension-one nilpotent ideal isomorphic to $\mfn_{28}$ that admit an exact $\G_2$-structure, and we distinguish those with special torsion of positive type or of negative type, respectively.
\subsection{The case $\mfn_{9}$}\label{subsec:n9}
Note first that the Lie algebra $\mathrm{Der}(\mfn_9)$ of all derivations of $\mfn_9$ is given by
\begin{equation}\label{eq:derivationsn9}
\begin{split}
  \mathrm{Der}(\mfn_{9})&=\left\{\left.\left(\begin{smallmatrix} f_{5,5}-f_{4,4} & 0 & 0 &0 & 0 & 0 \\
                                                    f_{4,3} & -f_{5,5}+2 f_{4,4} & 0 & 0 & 0 & 0 \\
                                                     0 & 0 & 2 f_{5,5}-2 f_{4,4} & 0 & 0 & 0 \\
                                                     f_{5,3} & f_{5,4} & f_{4,3} & f_{4,4} & 0 & 0 \\
                                                     f_{5,1} & f_{6,4} & f_{5,3} & f_{5,4} & f_{5,5} & 0 \\
                                                     f_{6,1} & f_{6,2} & f_{6,3} & f_{6,4} & f_{5,4} & 2 f_{5,5}-f_{4,4}
\end{smallmatrix}\right)\;\right|\; f_{i,j}\in\bR \right\}.
\end{split}
\end{equation}
with respect to the basis $(e_1,\ldots,e_6)$ of $\mfn_9$. This can be checked by a lengthy but straightforward calculation done efficiently with a computer algebra system like MAPLE. Exponentials of these derivations are then (inner) automorphisms of the Lie algebra $\mfn_9$. Using these automorphisms, one obtains:
\begin{lemma}\label{le:normalformsn9}
	Let $(\omega,\rho)\in \Lambda^2 \mfn_9^*\times \Lambda^3\mfn_9^*$ be an $\SU(3)$-structure on $\mfn_9$ with exact $\rho$, i.e.\ there exists some $\nu\in \Lambda^2 \mfn_9^*$ with $d\nu=\rho$. Then $\omega,\rho$ and $\nu$ are given, up to automorphism, by
	\begin{equation}\label{eq:normalformn9}
\left\{	\begin{split}
	\rho&=\epsilon (e^{125}+e^{146}-e^{236}-e^{345})=:\epsilon \rho_0,\\
	\omega&= a_1 e^{13}+a_2 e^{24}+a_3 e^{56}+a_4 (e^{12}+e^{34})+a_5 \left(e^{15}-e^{36}\right)+a_6 \left(e^{25}-e^{46}\right)+a_7 \left(e^{26}+e^{45}\right),\\
	\nu&=b_1 e^{12}+b_2 e^{13}+b_3 e^{14}+b_4 \left(e^{15}+e^{34}\right)+b_5\left(e^{16}+e^{35}\right) +b_6 e^{23}+b_7 e^{24}+\tfrac{\epsilon}{2}\left(e^{26}+e^{45}\right)+\epsilon\, e^{56},
	\end{split}\right.
	\end{equation}
	for certain $\epsilon\in \{1,-1\}$, $a_1,\ldots,a_7\in \bR$ with $a_1 a_2>0$, $a_1 a_3<0$ and certain $b_1,\ldots,b_7\in \bR$.
	If $(\omega,\rho)$ is half-flat and $\nu$ of type $(1,1)$, then we may assume that $a_4=a_5=0$, $b_1=b_4=0$ and $b_6=b_3$ in \eqref{eq:normalformn9}.
\end{lemma}
\begin{proof}
First of all, observe that the most general exact three-form $\rho$ is given by
\begin{equation*}
\rho=c_1 e^{123}+c_2 e^{124}+c_3 e^{125}+c_4 e^{126}+c_5 e^{134}+c_6 e^{135}+c_4 e^{145}+c_7 e^{146}+c_8 e^{234}-c_7 e^{236}-c_7 e^{345}
\end{equation*}
for certain $c_i\in \bR$, $i=1,\ldots,8$. The quartic invariant $\lambda(\rho)$ of $\rho$ computes to be equal to
\begin{equation*}
\lambda(\rho)=-4 c_7^2 (c_3 c_7-c_4^2) \left(e^{1234567}\right)^{\otimes 2}.
\end{equation*}
As $\lambda(\rho)$ has to be negative, we surely must have $c_7\neq 0$. Now exponentials of matrices as in \eqref{eq:derivationsn9} give automorphisms of $\mfn_9$. We consider first the automorphism $F_1:=\exp(A_1)$ for the matrix $A_1$ as in \eqref{eq:derivationsn9} with $f_{4,3}=0$, $f_{4,4}=0$, $f_{5,5}=0$, $f_{6,1}=0$, $f_{6,2}:=0$, $f_{6,3}=0$ and $f_{5,1}=-\frac{f_{5,3}\, f_{5,4}}{2}$. For notational simplicity, we set $a:=f_{5,3}$, $b:=f_{5,4}$, $s:=f_{6,4}$ and so have
\begin{equation*}
A_1:=\left(\begin{smallmatrix}
0 & 0& 0 & 0 & 0 & 0 \\
0 & 0& 0 & 0 & 0 & 0 \\
0 & 0& 0 & 0 & 0 & 0 \\
a & b& 0 & 0 & 0 & 0 \\
-\frac{ab}{2} & s& a & b & 0 & 0 \\
0 & 0 & 0 & s  &b & 0 \\
\end{smallmatrix}
\right)
\end{equation*}
The exponential of this matrix is easily computed to be
\begin{equation*}
F_1:=\left(\begin{smallmatrix}
1 & 0& 0 & 0 & 0 & 0 \\
0 & 1& 0 & 0 & 0 & 0 \\
0 & 0& 1 & 0 & 0 & 0 \\
a & b& 0 & 1 & 0 & 0 \\
0 & S& a & b & 1 & 0 \\
A & B & C & S  &b & 1 \\
\end{smallmatrix}
\right)
\end{equation*}
with $A=\tfrac{1}{2}as-\tfrac{1}{12}ab^2$, $B=bs+\tfrac{b^3}{6}$, $C=\tfrac{ab}{2}$ and $S=s+\tfrac{1}{2}b^2$. Then
\begin{equation*}
(F_1^*\rho)(e_1,e_2,e_6)=c_4+bc_7,\qquad (F_1^*\rho)(e_1,e_3,e_5)= c_6+a c_7,\qquad
(F_1^*\rho)(e_2,e_3,e_4)=c_8-2 s c_7
\end{equation*}
So, setting $a:=-\tfrac{c_6}{c_7}$ and $b:=-\tfrac{c_4}{c_7}$ and $s:=\tfrac{c_8}{2 c_7}$, we get that $(F_1^*\rho)(e_1,e_4,e_5)=(F_1^*\rho)(e_1,e_2,e_6)=0$, $(F_1^*\rho)(e_1,e_3,e_5)=0$ and $(F_1^*\rho)(e_2,e_3,e_4)=0$, i.e.\ we have
\begin{equation*}
F_1^*\rho=\tilde{c}_1 e^{123}+\tilde{c}_2 e^{124}+\tilde{c}_3 e^{125}+\tilde{c}_4 e^{134}+\tilde{c}_5\left( e^{146}-e^{236}- e^{345}\right),
\end{equation*}
for certain $\tilde{c}_i\in \bR$, $i=1,\ldots,5$, now with $\tilde{c}_5\neq 0$ (in fact, $\tilde{c}_5=c_7$). Next, we consider the automorphism $F_2:=\exp(A_2)$ with the matrix $A_2$ as in \eqref{eq:derivationsn9} with $f_{4,3}:=0$, $f_{4,4}=0$, $f_{5,5}=0$, $f_{5,3}=0$, $f_{5,4}=0$, $f_{5,1}=0$ and $f_{6,4}=0$, i.e.\ we have
\begin{equation*}
A_2:=\left(\begin{smallmatrix}
0 & 0& 0 & 0 & 0 & 0 \\
0 & 0& 0 & 0 & 0 & 0 \\
0 & 0& 0 & 0 & 0 & 0 \\
0 & 0& 0 & 0 & 0 & 0 \\
0 & 0& 0 & 0 & 0 & 0 \\
p & q & r & 0  &0 & 0 \\
\end{smallmatrix}
\right),
\end{equation*}
where we set $p:=f_{6,1}$, $q:=f_{6,2}$ and $r:=f_{6,3}$. With $F_2:=\exp(A_2)=I_6+A_2$, we obtain
\begin{equation*}
(F_2^*F_1^*\rho)(e_1,e_2,e_3)=\tilde{c}_1-\tilde{c}_5\, p,\qquad
(F_2^*F_1^*\rho)(e_1,e_2,e_4)=\tilde{c}_2-\tilde{c}_5\, q,\qquad
(F_2^*F_1^*\rho)(e_1,e_3,e_4)=\tilde{c}_4-\tilde{c}_5\, r.
\end{equation*}
Thus, setting $p:=\tfrac{\tilde{c}_1}{\tilde{c}_5}$, $q:=\tfrac{\tilde{c}_2}{\tilde{c}_5}$ and $r:=\tfrac{\tilde{c}_4}{\tilde{c}_5}$, we get that
\begin{equation*}
F_2^*F_1^*\rho= A e^{125}+B \left( e^{146}-e^{236}-e^{345}\right).
\end{equation*}
Then $\lambda(F_2^*F_1^*\rho)=-4 A B^3$, i.e.\ we must have $A\cdot B>0$. It is now fairly easy to see that the exponential $F_3:=\exp(A_3)$ of a diagonal matrix $A_3$ as in \eqref{eq:derivationsn9} allows to normalise $A=B\in \{-1,1\}$, i.e.\ we have
\begin{equation*}
(F_3^*F_2^*F_1^*\rho)=\epsilon(e^{125}+e^{146}-e^{236}- e^{345})=\epsilon\rho_0
\end{equation*}
for some $\epsilon\in \{-1,1\}$. 

From now on, we will use again $\rho$ for $F_3^*F_2^*F_1^*\rho$, so that $\rho=\epsilon \rho_0$. Note that $\epsilon (e_1,-e_3,e_2,-e_4,e_5,e_6)$ or $\epsilon (e_1,e_3,e_2,e_4,e_5,-e_6)$ is an oriented adapted basis for $\rho=\epsilon \rho_0$, depending on the orientation induced by $\omega^3$. Hence, the induced almost complex structure $J=J_{\rho}$ is either given by $J_0 e_1=-e_3$, $J_0 e_2=-e_4$ and $J_0 e_5=e_6$ or by $-J_0$. A straightforward computation shows that a two-form $\nu\in \Lambda^2 \mfn_9^*$ with $d\nu=\rho$ has to be as claimed and $\nu$ is of type $(1,1)$ precisely when $b_1=b_4=0$ and $b_6=b_3$. 

Next, we are interested in bringing $\omega$ into a canonical form.
 To this aim, note first that $\omega$ has to be a $(1,1)$-form with respect to $J_0$ and so
\begin{equation*}
\begin{split}
\omega\ =& a_1 e^{13}+a_2 e^{24}+a_3 e^{56}+a_4 \left(e^{12}+e^{34}\right)+a_5 \left(e^{14}+e^{23}\right)+a_6 \left(e^{15}-e^{36}\right) \\[-2pt]
&+a_7 \left(e^{16}+e^{35}\right)+a_8 \left(e^{25}-e^{46}\right)+a_9 \left(e^{26}+e^{45}\right)
\end{split}
\end{equation*}
for certain $a_1,\ldots,a_9\in \bR$. Observe that $a_1=\omega(e_1,e_3)=\omega(e_1,\mp J e_1)=\pm g(e_1,e_1)$, $a_2=\omega(e_2,e_4)=\pm g(e_2,e_2)$ and $a_3=\omega(e_5,e_6)=\mp g(e_5,e_5)$, and so $a_1 a_2>0$, $a_1 a_3<0$ as claimed.

In order to bring $\omega$ into a form with less parameters without changing $\rho$, we need to look at those matrices $A$ in \eqref{eq:derivationsn9} that are in the Lie algebra $\mfa$ of the stabiliser group of $\rho_0$.
Such an $A$ has to commute with $J_0$, which is the case if and only if
$f_{5,5}=f_{4,4}$, $f_{5,3}=0$, $f_{5,4}=0$, $f_{6,1}=0$, $f_{6,3}=-f_{5,1}$, $f_{6,2}=0$, $f_{6,4}=0$. Moreover, the complex $\pm J_0$-trace of $A$ must be equal to zero, which additionally gives us $f_{4,4}=0$. So $A$ is given by
\begin{equation*}
A=\left(\begin{smallmatrix} 0 & 0 & 0 &0 & 0 & 0 \\
x & 0 & 0 & 0 & 0 & 0 \\
0 & 0 & 0 & 0 & 0 & 0 \\
0& 0 & x & 0 & 0 & 0 \\
y & 0 &0 & 0 & 0 & 0 \\
0 & 0 & -y & 0 & 0 & 0
\end{smallmatrix}\right)
\end{equation*}
for $x:=f_{4,3}$ and $y:=f_{5,1}$. Then $F:=\exp(A)=I_6+A$ and
\begin{equation*}
(F^*\omega)(e_1,e_4)=a_5+x a_2-y a_9,\qquad (F^*\omega)(e_1,e_6)=a_7+x a_9+y a_3.
\end{equation*}
Now $g(e_2,e_2)=\pm a_2$, $g(e_5,e_5)=\mp a_3$, as we observed above, and $g(e_2,e_5)= \omega(e_2,Je_5)=\pm\omega(e_2,e_6)=\pm a_9$. Since any minor of $g$ has to be non-zero, we get
$0\neq -(g(e_2,e_2)g(e_5,e_5)-g(e_2,e_5)^2)=a_2 a_3+a_9^2$. Thus, setting
\begin{equation*}
x:=-\frac{a_3 a_5+a_7 a_9}{a_2 a_3+a_9^2},\qquad y:=-\frac{a_2 a_7-a_5 a_9}{a_2 a_3+a_9^2}
\end{equation*}
yields $(F^*\omega)(e_2,e_3)=(F^*\omega)(e_1,e_4)=0$ and $(F^*\omega)(e_3,e_5)=(F^*\omega)(e_1,e_6)=0$. Hence, renaming $F^*\omega$ by $\omega$ and using again coefficients labeled $a_1,\ldots,a_7$, we have
\begin{equation*}
\omega=a_1 e^{13}+a_2 e^{24}+a_3 e^{56}+a_4 (e^{12}+e^{34})+a_5 \left(e^{15}-e^{36}\right)+a_6 \left(e^{25}-e^{46}\right)+a_7 \left(e^{26}+e^{45}\right).
\end{equation*}
Then
\begin{equation*}\ts
d\left(\frac{1}{2}\omega^2\right)=(a_2 a_5+a_4 a_7+(a_5 a_7-a_3 a_4))e^{123435}+(a_5 a_7-a_3 a_4)e^{12356},
\end{equation*}
i.e.\ $(\omega,\rho)$ is half-flat if and only if
\begin{equation*}
a_2 a_5+a_4 a_7=0,\qquad a_5 a_7-a_3 a_4=0.
\end{equation*}
The first equation gives us $a_5=-\tfrac{a_4 a_7}{a_2}$ and inserting this into the second equation yields
\begin{equation*}
0= a_5 a_7-a_3 a_4=-\frac{a_4 a_7^2}{a_2}-\frac{a_2 a_3 a_4}{a_2}=-\frac{a_4(a_2a_3+a_7^2)}{a_2}
\end{equation*}
Since $a_7$ plays the role of the former $a_9$, we showed above that $a_2 a_3+a_7^2\neq 0$. Thus,
$a_4=0$ and so $a_5=0$. The equations $a_2 a_5+a_4 a_7=0$, $a_5 a_7-a_3 a_4=0$ are now fulfilled, so this finishes the proof.
\end{proof}
For the rest of this subsection, we assume that $(\omega,\rho,\nu)$ is as in \eqref{eq:normalformn9} with $(\omega,\rho)$ being half-flat and $\nu$ being a primitive form of type $(1,1)$. 
Moreover, we assume that \eqref{eq:1} and \eqref{eq:2} are valid. We show that these assumptions give rise to a contradiction. First note the following:
\begin{itemize}
\item[(i)]
We may 
assume that $\omega^3\in \bR_+\cdot e^{123456}$, i.e.\ $\omega$ induces the orientation in which the ordered basis $(e_1,\ldots,e_6)$ is oriented. This follows from the observation that with $(\omega,\rho,\nu,f)$ also $(-\omega,\rho,\nu,-f)$ satisfies \eqref{eq:1} and \eqref{eq:2}.
\item[(ii)]
Moreover, we may assume that $\epsilon=1$ as with $(\omega,\rho,\nu,f)$ also $(\omega,-\rho,-\nu,-f)$ fulfills \eqref{eq:1} and \eqref{eq:2}
\end{itemize}
Using these simplifications, we obtain:
\begin{lemma}\label{le:someequalitiesn9}
We have
\begin{equation*}
\begin{split}
f_{5,3}&=f_{6,1}=f_{6,4}=0,\qquad f_{6,3}=-f_{5,1},\qquad f_{6,2}=2 f_{5,4},\\
a_3&=f_{4,4}-3 f_{5,5},\qquad a_6=f_{5,4},\qquad a_7=-\frac{f_{4,4}+f_{5,5}}{2}
\end{split}
\end{equation*}
and $b_3=b_5=0$ or $f_{5,4}=0$,
\end{lemma}
\begin{proof}
	First of all,
	\begin{equation*}
	0=(\omega-f.\nu)(e_3,e_6)=f_{5,3}, \quad 0=(\omega-f.\nu)(e_3,e_4)=b_5 f_{5,4}-\tfrac{f_{5,3}}{2},\quad
	0=(\omega-f.\nu)(e_1,e_5)=b_5 f_{5,4}-f_{6,1}+\tfrac{f_{5,3}}{2},
	\end{equation*}
i.e. $f_{5,3}=f_{6,1}=0$ and $b_5f_{5,4}=0$. Moreover, we get
\begin{equation*}
\begin{split}
0&=(\omega-f.\nu)(e_3,e_5)=f.\nu= -f_{6,3}+\tfrac{f_{4,3}}{2}+b_5(3 f_{5,5}-2 f_{4,4}),\\
0&=(\omega-f.\nu)(e_1,e_6)=f.\nu= f_{5,1}+\tfrac{f_{4,3}}{2}+b_5(3 f_{5,5}-2 f_{4,4}),
\end{split}
\end{equation*}
which yields $f_{6,3}=-f_{5,1}$. Furthermore,
\begin{equation*}
\begin{split}
0&=(\omega-f.\nu)(e_2,e_6)=a_7+\frac{f_{4,4}+f_{5,5}}{2}+f_{6,4},\\
0&=(\omega-f.\nu)(e_4,e_5)=a_7+\frac{f_{4,4}+f_{5,5}}{2}- f_{6,4},
\end{split}
\end{equation*}
which gives $f_{6,4}=0$ as well as $a_7=-\tfrac{f_{4,4}+f_{5,5}}{2}$.
Next, we have
\begin{equation*}
0=(\omega-f.\nu)(e_2,e_5)=a_6- f_{6,2}+f_{5,4},\qquad 0= (\omega-f.\nu)(e_4,e_6)=-a_6+f_{5,4},
\end{equation*}
i.e.\ $f_{6,2}=2 f_{5,4}$ and $a_6= f_{5,4}$. Moreover, we get
\begin{equation*}
\begin{split}
0=(\omega-f.\nu)(e_5,e_6)=a_3-(f_{4,4}-3f_{5,5}),\qquad 0=(\omega-f.\nu)(e_1,e_2)=f_{5,4}(b_3+2b_5),
\end{split}
\end{equation*}
i.e. $a_3=f_{4,4}-3f_{5,5}$, and, since also $b_5f_{5,4}=0$, $f_{5,4}=0$ or $b_3=b_5=0$.
\end{proof}

\begin{lemma}\label{le:f54=0}
	In Lemma \ref{le:someequalitiesn9}, we must have $f_{5,4}=0$.
	\end{lemma}
\begin{proof}
Assume that $f_{5,4}\neq 0$. Then $b_3=b_5=0$ by Lemma \ref{le:someequalitiesn9} and so
\begin{equation*}
\begin{split}
0&=(\omega-f.\nu)(e_1,e_5)=\tfrac{f_{4,3}}{2}+f_{5,1},\qquad 0=(\omega-f.\nu)(e_1,e_3)=a_1+3 b_2 (f_{5,5}-f_{4,4}),\\
 0&=(\omega-f.\nu)(e_2,e_4)=a_2+b_7(3f_{4,4}-f_{5,5}),
 \end{split}
\end{equation*}
i.e. $a_1=3 b_2 (f_{4,4}-f_{5,5})$, $a_2=b_7(f_{5,5}-3f_{4,4})$ and $f_{5,1}=-\frac{f_{4,3}}{2}$.
Imposing these identities, we get
\begin{equation*}
0=(\omega-f.\nu)(e_1,e_4)=\frac{f_{4,3}(1+4 b_7)}{4},
\end{equation*}
and so either $f_{4,3}=0$ or $b_7=-\tfrac{1}{4}$ holds.

We show that $f_{4,3}=0$ and argue by contradiction, i.e.\ we assume that $f_{4,3}\neq 0$ and so $b_7=-\tfrac{1}{4}$. Then \eqref{eq:2} gives us
\begin{equation*}
0=(f.\omega\wedge \omega-\omega\wedge\nu-d\hat{\rho})(e_1,e_4,e_5,e_6)=-f_{4,3}( (f_{4,4}-f_{5,5})^2+f_{5,4}^2),
\end{equation*}
so that $f_{5,5}=f_{4,4}$ and $f_{5,4}=0$ due to $f_{4,3}\neq 0$. But then
one checks that $\omega^3=0$, a contradiction. Hence, we must have $f_{4,3}=0$.

Assuming $f_{4,3}=0$, one computes
\begin{equation}\label{eq:furthereqns}
\begin{split}
0&=(f.\omega\wedge \omega-\omega\wedge\nu-d\hat{\rho})(e_1,e_2,e_3,e_5)=-\epsilon b_2 f_{5,4} \left(12 f_{4,4}^2- 36 f_{4,4} f_{5,5}+24 f_{5,5}^2-1\right) \\
0&=(f.\omega\wedge \omega-\omega\wedge\nu-d\hat{\rho})(e_1,e_3,e_5,e_6)=2 b_2 (3 f_{4,4}^2 - 12 f_{4,4} f_{5,5} + 9 f_{5,5}^2 - 1)(2f_{4,4} - 3f_{5,5})
\end{split}
\end{equation}
One checks that $b_2=0$ implies $\omega^3= 0$, and so we must have $b_2\neq 0$. Since $f_{5,4}\neq 0$ by assumption, \eqref{eq:furthereqns} yields
\begin{equation*}
\begin{split}
0&=12 f_{4,4}^2- 36 f_{4,4} f_{5,5}+24 f_{5,5}^2-1,\\
0&=(3 f_{4,4}^2 - 12 f_{4,4} f_{5,5} + 9 f_{5,5}^2 - 1)(2f_{4,4} - 3f_{5,5})
\end{split}
\end{equation*}
So either $3 f_{4,4}^2 - 12 f_{4,4} f_{5,5} + 9 f_{5,5}^2 - 1=0$ or $2f_{4,4}-3f_{5,5}=0$.
However, both cases give us a contradiction:

Namely, if $3 f_{4,4}^2 - 12 f_{4,4} f_{5,5} + 9 f_{5,5}^2 - 1=0$, then
\begin{equation*}
3 f_{4,4}^2 - 12 f_{4,4} f_{5,5} + 9 f_{5,5}^2=1=12 f_{4,4}^2- 36 f_{4,4} f_{5,5}+24 f_{5,5}^2,
\end{equation*}
and so
\begin{equation*}
0=9 f_{4,4}^2 - 24 f_{4,4} f_{5,5} + 15 f_{5,5}^2=(3f_{4,4}-4 f_{5,5})^2-f_{5,5}^2,
\end{equation*}
i.e.
\begin{equation*}
3 f_{4,4}-4 f_{5,5}=\pm f_{5,5}.
\end{equation*}
So either $f_{4,4}-f_{5,5}=0$ or $f_{5,5}=\frac{3}{5} f_{4,4}$. However, in the first case, one checks that $\omega^3=0$, a contradiction. Thus, we must have $f_{5,5}=\frac{3}{5} f_{4,4}$. But then
\begin{equation*}\ts
1=3 f_{4,4}^2 - 12 f_{4,4} f_{5,5} + 9 f_{5,5}^2= 3 f_{4,4}^2-\frac{36}{5} f_{4,4}^2+\frac{81}{25} f_{4,4}^2=-\frac{24}{25} f_{4,4}^2\leq 0,
\end{equation*}
again a contradiction. 

Consider now the case $2f_{4,4}-3f_{5,5}=0$. Then $f_{5,5}=\frac{2}{3} f_{4,4}$ and so
\begin{equation*}
\begin{split}
1&=\ts 12 f_{4,4}^2- 36 f_{4,4} f_{5,5}+24 f_{5,5}^2=-\frac{4}{3} f_{4,4}^2\leq 0,
\end{split}
\end{equation*}
which is a also a contradiction.

Thus, we must have $f_{5,4}=0$.
\end{proof}
To simplify the notation, we set from now on
\begin{equation*}
x:=f_{4,4},\qquad y:=f_{5,5},\qquad z:=f_{4,3}.
\end{equation*}
With this new notation, one gets:
\begin{lemma}\label{le:n9no3}
We have $a_2=(y-3 x) b_7$, $f_{5,1}=(2x-3y)b_5-\tfrac{z}{2}$, $x+y\neq 0$, $4 b_7+1\neq 0$, $b_2\neq 0$ and $a_1=\tfrac{b_2}{2(x+y)}$.
\end{lemma}
\begin{proof}
Firstly
\begin{equation*}
0=(\omega-f.\nu)(e_2,e_4)=a_2+b_7 (3x-y) ,\quad 0=(\omega-f.\nu)(e_3,e_5)=(3y-2x)b_5 +\tfrac{z}{2} + f_{5,1},
\end{equation*}
i.e.\ $a_2=(y-3 x) b_7$, $f_{5,1}=(2x-3y)b_5-\tfrac{z}{2}$. Next, set
\begin{equation*}
A:=(12 b_7-1)x^2-(40 b_7+2) xy+(12 b_7-1)y^2.
\end{equation*}
Then
\begin{equation*}
0\neq \omega^3= \tfrac{3}{2} a_1 A e^{123456},
\end{equation*}
which implies $a_1\neq 0$ and $A\neq 0$. Since
\begin{equation*}
0= \nu\wedge\omega^2= \big(\tfrac{A b_2}{2}+a_1 (4b_7+1) (x+y) \big)e^{123456},
\end{equation*}
we either have $b_2=0$, and then $(4b_7+1) (x+y)=0$ as well, or $b_2\neq 0$, and so $x+y\neq 0$, $4b_7+1\neq 0$, and $A=-\tfrac{2 a_1 (4b_7+1)(x+y)}{b_2}$. Moreover, we have
\begin{equation}\label{eq:A}
0= \left(f.\omega\wedge \omega-d\hat{\rho}-\nu\wedge \omega\right)(e_2,e_4,e_5,e_6)=\tfrac{(x+y)(A+4b_7+1)}{2}
\end{equation}
Assume now first that $b_2=0$. Then we must have $x+y=0$, as otherwise $4b_7+1=0$ and so $(x+y)A=0$, a contradiction to $x+y\neq 0$ and $A\neq 0$. But then
\begin{equation*}
\begin{split}
0 &= \left(f.\omega\wedge \omega-d\hat{\rho}-\nu\wedge \omega\right)(e_1,e_2,e_3,e_6)=-1+\tfrac{a_1}{2},\\
0&=\left(f.\omega\wedge \omega-d\hat{\rho}-\nu\wedge \omega\right)(e_1,e_3,e_5,e_6)=a_1 (40x^2-1),
\end{split}
\end{equation*}
from which we obtain $a_1=2$ and $x=\delta \sqrt{\tfrac{1}{40}}$ for some $\delta\in \{-1,1\}$. But then
\begin{equation*}
0 = \left(f.\omega\wedge \omega-d\hat{\rho}-\nu\wedge \omega\right)(e_1,e_2,e_3,e_4)=\tfrac{12}{5}b_7-3,
\end{equation*}
i.e.\ $b_7=\tfrac{5}{4}$. Hence,
\begin{equation*}
\begin{split}
0 &= \left(f.\omega\wedge \omega-d\hat{\rho}-\nu\wedge \omega\right)(e_1,e_4,e_5,e_6)=\tfrac{z}{2}-\delta \tfrac{\sqrt{10}}{5}b_3,\\
0&=\left(f.\omega\wedge \omega-d\hat{\rho}-\nu\wedge \omega\right)(e_2,e_3,e_4,e_5)=-\tfrac{\delta}{8}(3\sqrt{10} b_5 - 2\delta z),
\end{split}
\end{equation*}
and so $b_3= \tfrac{5\delta}{2 \sqrt{10}} z$, $b_5=\tfrac{2\delta}{3\sqrt{10}} z$. However,
\begin{equation*}
0=(\omega-f.\nu)(e_1,e_4)=\tfrac{31}{24}z^2,
\end{equation*}
i.e.\ $z=0$, and so
\begin{equation*}
0=(\omega-f.\nu)(e_1,e_3)=2,
\end{equation*}
a contradiction.

Hence, we must have $b_2\neq 0$, $x+y\neq 0$, $4b_7+1\neq 0$ and $A=-\tfrac{2 a_1 (4b_7+1)(x+y)}{b_2}$. But then \eqref{eq:A} gives us $A=-(4b_7+1)$. Thus,
\begin{equation*}
\tfrac{2 a_1 (4b_7+1)(x+y)}{b_2}=-A=4b_7+1
\end{equation*}
and so, since $4b_7+1\neq 0$,
\begin{equation*}
a_1=\tfrac{b_2}{2(x+y)}.
\end{equation*}
\end{proof}
This allows us now to prove:
\begin{theorem}\label{th:noeigenformsn9}
	Let $\mfg$ be a seven-dimensional almost nilpotent Lie algebra with codimension-one nilpotent ideal isomorphic to $\mfn_9$. Then
	$\mfg$ does not admit a closed $\G_2$-eigenform.
\end{theorem}
\begin{proof}
We assume that the parameters fulfill all the conditions that we derived in all the previous lemmas. Then we first get
\begin{equation*}
0=(\omega-f.\nu)(e_2,e_3)=-\tfrac{(4x-6y)b_5-4y b_3-(4b_7+1)z}{4},
\end{equation*}
i.e.
\begin{equation*}
z=\tfrac{(4x-6y)b_5-4y b_3}{4b_7+1}
\end{equation*}
since $4b_7+1\neq 0$ by Lemma \ref{le:n9no3}. Then one computes
\begin{equation*}
0 = \left(f.\omega\wedge \omega-d\hat{\rho}-\nu\wedge \omega\right)(e_1,e_2,e_3,e_6)=-\tfrac{b_2(6 x y+6y^2-1)+4(x+y)}{4(x+y)},\\
\end{equation*}
i.e.\ $b_2(6 x y+6y^2-1)=-4(x+y)$. Since $x+y\neq 0$, also $6xy+6y^2-1\neq 0$, and so
\begin{equation*}
b_2=-\tfrac{4(x+y)}{6xy+6y^2-1}.
\end{equation*}
Moreover,
\begin{equation*}
\begin{split}
  0&=\left(f.\omega\wedge \omega-d\hat{\rho}-\nu\wedge \omega\right)(e_2,e_4,e_5,e_6)
  =\tfrac{x+y}{2}\, \big(b_7(12x^2-40xy+12y^2+4)+1-(x+y)^2\big),\\
0&=\left(f.\omega\wedge \omega-d\hat{\rho}-\nu\wedge \omega\right)(e_1,e_3,e_5,e_6)=-\tfrac{b_2}{2(6xy+6y^2-1)}\, (2x^2 - 14xy + 24y^2 - 1),
\end{split}
\end{equation*}
that is,
\begin{equation*}
2x^2 - 14xy + 24y^2 - 1=0,\qquad b_7(12x^2-40xy+12y^2+4)=(x+y)^2-1,
\end{equation*}
since $b_2\neq 0$, $x+y\neq 0$ by Lemma \ref{le:n9no3}. 

We show now that $12x^2-40xy+12y^2+4\neq 0$:

If this is not the true, then $(x+y)^2=1$, i.e.\ $y=\delta-x$
for some $\delta\in \{-1,1\}$. But then $0=12x^2-40xy+12y^2+4=16\left(2x-\delta\right)^2$, i.e.\ $x=\tfrac{\delta}{2}=y$ and so $2x^2 - 14xy + 24y^2 - 1=2\neq 0$, a contradiction.

Thus, $12x^2-40xy+12y^2+4\neq 0$ and we have
\begin{equation*}
 b_7=\tfrac{(x+y)^2-1}{12x^2-40xy+12y^2+4}.
\end{equation*}
On then computes
\begin{equation*}
0 = \left(f.\omega\wedge \omega-d\hat{\rho}-\nu\wedge \omega\right)(e_2,e_3,e_5,e_6)=\tfrac{(b_5+2 b_3)(x-4y)}{2}.
\end{equation*}
Hence, $b_5=-2b_3$ or $x=4y$. However, $x=4y$ is impossible since then
\begin{equation*}
0 = \left(f.\omega\wedge \omega-d\hat{\rho}-\nu\wedge \omega\right)(e_1,e_3,e_5,e_6)=\tfrac{2}{30y^2-1},
\end{equation*}
a contradiction. Thus $b_5=-2b_3$ and one gets
\begin{equation*}
0 = \left(f.\omega\wedge \omega-d\hat{\rho}-\nu\wedge \omega\right)(e_1,e_2,e_4,e_6)=-\tfrac{2 b_3 (x-y)^2 (x+2y)}{3x^2-10xy+3y^2+1}
\end{equation*}
and
\begin{equation*}
0 \neq \omega^3= \tfrac{12 (x-y)^2}{(6xy+6y^2-1)\cdot (3x^2-10xy+3y^2+1)}e^{123456}.
\end{equation*}
Thus, $x-y\neq 0$ and so $b_3(x+2y)=0$. We show that $b_3\neq 0$ and, consequently, $x=-2y$:

If $b_3=0$, then $b_5=0$ as well and we do get
\begin{equation*}
\begin{split}
0&=(\omega-f.\nu)(e_1,e_3)=\tfrac{12x^2-12y^2-2}{6xy+6y^2-1},\\
0&= \left(f.\omega\wedge \omega-d\hat{\rho}-\nu\wedge \omega\right)(e_1,e_3,e_5,e_6)=\tfrac{-4x^2+28xy-48y^2+2}{6xy+6y^2-1}.
\end{split}
\end{equation*}
One easily checks that all solutions of
\begin{equation*}
-12x^2+12y^2+2=0,\qquad -4x^2+28xy-48y^2+2=0
\end{equation*}
are given by
\begin{equation*}
(x,y)=\delta_1 \tfrac{\sqrt{30}}{15} \left(\tfrac{3}{2},1\right),\qquad (x,y)=\tfrac{\delta_2}{12} (5,-1)
\end{equation*}
for $\delta_1,\delta_2\in \{-1,1\}$. However, in the first case, one computes
\begin{equation*}
0= \left(f.\omega\wedge \omega-d\hat{\rho}-\nu\wedge \omega\right)(e_1,e_2,e_3,e_4)=-\tfrac{98}{27}
\end{equation*}
and in the second case one obtains
\begin{equation*}
0= \left(f.\omega\wedge \omega-d\hat{\rho}-\nu\wedge \omega\right)(e_1,e_2,e_3,e_4)=-\tfrac{365}{119},
\end{equation*}
and so a contradiction in both cases.

Hence, $b_3\neq 0$ and so $x=-2y$. But then
\begin{equation*}
0 = \left(f.\omega\wedge \omega-d\hat{\rho}-\nu\wedge \omega\right)(e_1,e_3,e_5,e_6)=\tfrac{120y^2-2}{6y^2+1},
\end{equation*}
i.e.\ $y=\delta \sqrt{\tfrac{1}{60}}$ for some $\delta\in \{-1,1\}$ and we finally obtain
\begin{equation*}
0 = \left(f.\omega\wedge \omega-d\hat{\rho}-\nu\wedge \omega\right)(e_1,e_2,e_3,e_4)=-\tfrac{686}{209},
\end{equation*}
a contradiction.

Hence, $\mfg$ does not admit a closed $\G_2$-eigenform.
\end{proof}
\subsection{The case $\mfn_{28}$}\label{subsec:n28}
In this subsection, we are considering exact $\G_2$-structures and closed $\G_2$-eigenforms on seven-dimensional almost nilpotent Lie algebras with codimension-one nilpotent ideal isomorphic 
to $\mfn_{28}$. We will determine all such Lie algebras which admit an exact $\G_2$-structure and we will show that no such Lie algebra can admit a closed $\G_2$-eigenform.

First of all, note that $\mfn_{28}$ is a well-known real six-dimensional nilpotent Lie algebra, namely the one underlying the complex three-dimensional Heisenberg Lie algebra, and the Iwasawa manifold, see e.g.\ \cite{KSa}. Moreover, for this subsection, denote by $J_0$ the almost complex structure on $\mfn_{28}$ uniquely defined by $J_0e_{2i-1}=e_{2i}$ for $i=1,2$ and $J_0 e_5=-e_6$.

The Lie algebra of all derivations of $\mfn_{28}$ is given by
\begin{equation}\label{eq:derivationsn28}
\mathrm{Der}(\mfn_{28})=\left\{\left.\begin{pmatrix} A & 0 \\ B & \trC A \end{pmatrix}\right|A\in \bC^{2\times 2},\, B\in \bR^{2\times 4}\right\},
\end{equation}
with respect to the basis $(e_1,\ldots,e_6)$, where we consider $A\in \bC^{2\times 2}$ as a real $4\times 4$-matrix and $\trC A\in \bC$ as a real $2\times 2$-matrix. The Lie group $\mathrm{Inn}(\mfn_{28})$ of inner automorphism of $\mfn_{28}$ is given by
the Lie group generated by the exponentials of elements in $\mathrm{Der}(\mfn_{28})$, and so equals
\begin{equation}\label{eq:innerautosn28}
\mathrm{Inn}(\mfn_{28})=\left\{\left.\begin{pmatrix} C & 0 \\ D & \det_{\bC}(C) \end{pmatrix}\right|C\in \GL(2,\bC),\, D\in \bR^{2\times 4}\right\}.
\end{equation}
Now split $\mfn_{28}=V\oplus W$ with $V:=\spa{e_1,e_2,e_3,e_4}$ and $W:=\spa{e_5,e_6}$ and do the same for the dual space $\mfn_{28}^*=V^*\oplus W^*$.
Let an $\SU(3)$-structure $(\omega,\rho)\in \Lambda^2 \mfn_{28}^*\times \Lambda^3\mfn_{28}^*$ with exact $\rho$ be given, i.e.\ there exists
$\nu \in \Lambda^2\mfn_{28}^*$ with $d\nu=\rho$. Write
\begin{equation*}
\omega=a e^{56}+e^5\wedge \alpha_1+e^6\wedge \alpha_2+\tilde{\omega}
\end{equation*}
for $a\in \bR$, $\alpha_1,\alpha_2\in V^*$ and $\tilde{\omega}\in \Lambda^2 V^*$ and, similarly,
\begin{equation*}
\nu=b e^{56}+e^5\wedge \beta_1+e^6\wedge \beta_2+\tilde{\nu}
\end{equation*}
for $b\in \bR$, $\beta_1,\beta_2\in V^*$ and $\tilde{\nu}\in \Lambda^2 V^*$. Moreover, let $\sigma_2:=de^5$ and $\sigma_3:=de^6$, note that $\sigma_2,\sigma_3\in [[\Lambda^{2,0} V^*]]$ with respect to the almost complex structure $J_0$ on $V$. Set
\begin{equation*}
\rho_0:=e^6\wedge \sigma_2-e^5\wedge \sigma_3.
\end{equation*}
Then $\rho_0$ induces the complex structure $J_0$ (if one chooses the right orientation on $\mfn_{28}$) and we have
\begin{equation*}
\rho=d\nu=b\rho_0+\sigma_2\wedge \beta_1+\sigma_3\wedge \beta_2.
\end{equation*}
This shows that $b\neq 0$ as otherwise $\rho(e_5,\cdot,\cdot)=0$, contradicting Lemma \ref{le:SL3Clinearlydependent} (a). Moreover,
$\rho(e_5,e_6,\cdot)=0$, i.e.\ $e_5$ and $e_6$ are $J:=J_{\rho}$-linearly dependent by Lemma \ref{le:SL3Clinearlydependent} (a).
But so
\begin{equation*}
0\neq g(e_5,e_5)=\omega(Je_5,e_5)=a e^{56}(Je_5,e_5),
\end{equation*}
implies $a\neq 0$ and we may apply an inner automorphism $F$ as in \eqref{eq:innerautosn28} with $C=I_2$ and suitable $D\in \bR^{2\times 4}$
to get rid of $\alpha_1$ and $\alpha_2$ in $\omega$, i.e.\ we may assume that $\omega=a e^{56}+\tilde{\omega}$ with $\tilde{\omega}\neq 0$ due to the non-degeneracy of $\omega$. Now we must have
\begin{equation*}
0=\omega\wedge \rho=ae^{56}\wedge (\sigma_2\wedge \beta_1+\sigma_3\wedge \beta_2)+be^6\wedge\sigma_2\wedge \tilde{\omega}-b e^5\wedge\sigma_3\wedge\tilde{\omega}
\end{equation*}
Thus, $\sigma_2\wedge \beta_1+\sigma_3\wedge \beta_2=0$, and so $\rho=b\rho_0$, and $\sigma_i\wedge \tilde{\omega}=0$ for $i=2,3$.
Now $\sigma_1$, $\sigma_2$ span $[[\Lambda^{2,0} V^*]]$ and so the latter identity shows
$\tilde{\omega}\in [\Lambda^{1,1} V^*]$. A straightforward computation shows that $\sigma_3\wedge \beta_2=-\sigma_2\wedge \beta_1$ implies $\beta_2=J_0^*\beta_1$.

Next, $(\sigma_1,J_0|_V)$, $\sigma_1:=e^{12}+e^{34}$, defines an almost Hermitian structure on $V$ and
$\mathrm{SU}(2)$ preserves $\sigma_1$ and acts as $\mathrm{SO}(3)$ on $[\Lambda^{1,1}_0 V^*]$.
Since matrices of the block-diagonal form $\diag(A,I_2)$ with $A\in \mathrm{SU}(2)$ are in $\mathrm{Inn}(\mfn_{28})$ and preserve $\rho_0$, we may thus assume that $\tilde{\omega}=a_1 e^{12}+a_2 e^{34}$. Moreover, an element of the form $\diag\left(b_1,b_1,b_2,b_2,b_1 b_2,b_1 b_2\right)$ is in
$\mathrm{Inn}(\mfn_{28})$ and only scales $\rho_0$, and so we may even assume that $|a_1|=|a_2|=|a|$, i.e. there are $\epsilon_1,\epsilon_2\in \{-1,1\}$ such that $\tilde{\omega}=\epsilon_1 (a\epsilon_2 e^{12}+e^{34})$. Since $J=\pm J_0$, we do get
\begin{equation*}
\epsilon_2 a^2= \omega(e_1,e_2)\omega(e_3,e_4)=\omega(Je_2,e_2)\omega(Je_4,e_4)=g(e_2,e_2)g(e_4,e_4)>0,
\end{equation*}
and so $\epsilon_2=1$, i.e. $\omega=a\epsilon_1( e^{12}+e^{34})+ ae^{56}$. Moreover,
\begin{equation*}
\epsilon_1 a^2=\omega(e_1,e_2)\omega(e_5,e_6)=-\omega(Je_2,e_2)\omega(Je_6,e_6)=-g(e_2,e_2)g(e_6,e_6)<0,
\end{equation*}
i.e. $\epsilon_1=-1$ and, consequently, $\omega=a (-e^{12}+e^{34})+ e^{56})=:a\omega_0$. Noting that for $a>0$, the ordered basis $(e_1,-e_2,e_3,-e_4,e_5,-e_6)$ is oriented and so $J=J_0$, and otherwise $J=-J_0$,
the normalisation condition reads $|a|^3= b^2$. Hence, we may write $a=\epsilon\lambda^2$ and $b=\lambda^3$ for some $\lambda\in \bR\setminus \{0\}$ and some $\epsilon\in \{-1,1\}$. 

Finally, we may use block-diagonal matrices $\diag(A,I_2)$ with $A\in \mathrm{SU}(2)$ to bring $\tilde{\nu}$ into a canonical. For this, note that $\diag(A,I_2)$ is in $\mathrm{Inn}(\mfn_{28})$ and preserves $(\omega,\rho)$ and so we may assume that
$\tilde{\nu}=c_1 e^{12}+c_2 e^{34} + c_3 \sigma_2+c_4 \sigma_3$ for certain $c_1,c_2,c_3,c_4\in \bR$.
Thus,
\begin{equation*}
\nu=\lambda^3 e^{56}+e^5\wedge \beta+e^6\wedge J^*\beta +c_1 e^{12}+c_2 e^{34} + c_3 \sigma_2+c_4 \sigma_3
\end{equation*}
for $\beta:=\beta_1\in V^*$. If $\nu\in [\Lambda^{1,1}_0 \mfn_{28}^*]$, then
$c_3=c_4=0$ and $c_2=\lambda^3-c_1$. Summarizing, we have arrived at
\begin{lemma}\label{le:canonicalformsonn28}
	Let $(\omega,\rho)\in \Lambda^2 \mfn_{28}^*\times \Lambda^3 \mfn_{28}^*$ be an $\SU(3)$-structure for which there exists $\nu\in \Lambda^2 \mfn_{28}^*$
	with $\rho=d\nu$. Then $(\omega,\rho,\nu)$, are, up to automorphism, given by
	\begin{equation}\label{eq:canonicalformsonn28one}
		\left\{\begin{split}
			\omega&=\epsilon\lambda^2\omega_0=\epsilon\lambda^2\left(-e^{12}-e^{34}+e^{56}\right),\\
			\rho&=\lambda^3 \rho_0=\lambda^3\left(e^{136}-e^{246}-e^{145}-e^{235}\right),\\
			\nu&=\lambda^3 e^{56}+e^5\wedge \beta+e^6\wedge J^*\beta+c_1 e^{12}+c_2 e^{34} + c_3 \sigma_2+c_4 \sigma_3,
		\end{split}\right.
	\end{equation}
	for certain $c_1,c_2,c_3,c_4\in \bR$,$\lambda\in \bR\setminus \{0\}$, $\epsilon\in \{-1,1\}$ and $\beta\in V^*$.
	
	If $\nu\in [\Lambda^{1,1}_0 \mfn_{28}^*]$,
	then, up to an automorphism, $(\omega,\rho)$ take the form as in \eqref{eq:canonicalformsonn28one} and
	\begin{equation}\label{eq:canonicalformsonn28two}
		\nu=\lambda^3  e^{56}+e^5\wedge \beta+e^6\wedge J_0^*\beta+c e^{12}+(\lambda^3-c) e^{34}
	\end{equation}
	for some $c\in \bR$ and $\beta\in \mfn_{28}^*$.
\end{lemma}
Next, we determine those seven-dimensional almost nilpotent Lie algebras $\mfg$ with codimension-one nilpotent ideal $\mfn_{28}$ which admit an exact $\G_2$-structure. First of all, we get some restriction on $f$ if $\mfg$ admits an exact $\G_2$-structure, i.e.\ if \eqref{eq:3} is valid:
\begin{lemma}\label{le:implicationsofeq3}
Let $(\omega,\rho)\in \Lambda^2 \mfn_{28}^*\times \Lambda^3 \mfn_{28}^*$ be as in \eqref{eq:canonicalformsonn28one}
and assume that $\nu\in \Lambda^2 \mfn_{28}^*$ satisfies $d\nu=\rho$. Then $\alpha\in \mfn_{28}^*$ and $f\in \mathrm{Der}(\mfn_{28})$ fulfill \eqref{eq:3}
if and only if
\begin{equation*}
f.\nu^{1,1}=\omega,\qquad f.\nu^{2,0}=d\alpha,\qquad [f,J_0]=0,
\end{equation*}
where $\nu^{1,1}$ is the $(1,1)$-part and $\nu^{2,0}$ is the $(2,0)+(0,2)$-part of $\nu$. If this is the case, then no eigenvalue of $f$ is purely imaginary.
\end{lemma}
\begin{proof}
We decompose $f=f_1+f_2$ into its $J_0$-invariant part $f_1$ and its $J_0$-anti-invariant part $f_2$. Then $f_1$ preserves the splitting
$\Lambda^2 \mfn_{28}^*=[\Lambda^{1,1} \mfn_{28}^*]\oplus [[\Lambda^{2,0} \mfn_{28}^*]]$ while $f_2$ interchanges the two summands.
As $d\alpha$ is of type $(2,0)+(0,2)$, \eqref{eq:3} is equivalent to
\begin{equation*}
f_1.\nu^{1,1}+f_2.\nu^{2,0}=\omega,\qquad f_1.\nu^{2,0}+f_2.\nu^{1,1}=d\alpha.
\end{equation*}
if we decompose $\nu=\nu^{1,1}+\nu^{2,0}$ as in the statement. Note that $f_2$ is a strictly lower triangular block matrix with respect to the splitting $\mfn_{28}=V\oplus W$, while $f_1$ is a lower triangular block matrix with respect to the same splitting. Moreover, $\nu^{2,0},\ d\alpha\in \Lambda^2 V^*$
and $\nu^{1,1}$ has a non-trivial $\Lambda^2 W^*$-part. Thus, $f_1.\nu^{2,0}\in \Lambda^2 V^*$
and $f_2.\nu^{1,1}\in W^*\wedge V^*\oplus \Lambda^2 V^*$ with non-trivial $W^*\wedge V^*$-part if $f_2\neq 0$.
Hence, $f_2=0$, i.e.\ $f=f_1$, and so $[f,J_0]=0$, and the above equations simplify to
\begin{equation*}
f.\nu^{1,1}=\omega,\qquad f.\nu^{2,0}=d\alpha
\end{equation*}
as stated.

Finally, assume that there is an eigenvector $X\in \mfn_{28}\setminus \{0\}$ of $f$ with purely imaginary eigenvalue $ic$, $c\in \bR$. Then
$f(X)=c J_0 X$, and so $f(J_0 X)=J_0 f(X)=-c X$, and we get
\begin{equation*}
\begin{split}
0\neq & g(X,X)=\omega(J_0 X,X)=f.\nu^{1,1}(J_0 X,X)=-\nu^{1,1}(f(J_0 X),X)-\nu^{1,1}(J_0 X,f(X))\\
&=c\nu^{1,1}(X,X)-c \nu^{1,1}(J_0X,J_0 X)=0
\end{split}
\end{equation*}
a contradiction. Hence, no eigenvalue of $f$ can be purely imaginary.
\end{proof}
We are now in the position to give a classification of those almost nilpotent Lie algebras with codimen\-sion-one nilpotent ideal isomorphic to $\mfn_{28}$ which admit an exact $\G_2$-structure:
\begin{theorem}\label{th:exactG2n28}
Let $\mfg$ be a seven-dimensional almost nilpotent Lie algebras $\mfg$ with codimension-one nilpotent ideal
$\mfn_{28}$, i.e.\ $\mfg\cong\mfn_{28}\rtimes_{f} \bR$ for some $f\in \mathrm{Der}(\mfn_{28})$. Then $\mfg$ admits
an exact $\G_2$-structure if and only if $f$ has no purely imaginary eigenvalues.
Equivalently, $\mfg$ admits an exact $\G_2$-structure if and only if $\mfg\cong \mfn_{28}\rtimes_{f_{a,b_1,b_2}}\bR$ for certain $a\in \left[-\tfrac{1}{4},\infty\right)\setminus \{0\}$, $b_1,b_2\in \bR$ or
$\mfg \cong \mfn_{28}\rtimes_{h_b}\bR$ for some $b\in \bR$, where
\begin{equation*}
\begin{split}
f_{a,b_1,b_2}&:=\left(\begin{smallmatrix}  a+ib_1 & & \\ & -\tfrac{1}{2}-a+ib_2 & \\ & & -\tfrac{1}{2}+i(b_1+b_2) \end{smallmatrix}\right),\quad
h_b:=\left(\begin{smallmatrix} -\tfrac{1}{4}+i b & 1 &  \\  & -\tfrac{1}{4}+i b &  \\ &  & -\tfrac{1}{2} + 2 ib \end{smallmatrix}\right).
\end{split}
\end{equation*}
\end{theorem}
\begin{proof}
We note that the forward implication in the first statement is incorporated in Lemma \ref{le:implicationsofeq3}.

So let $\mfg\cong\mfn_{28}\rtimes_{f} \bR$ for some $f\in \mathrm{Der}(\mfn_{28})$ which has no purely imaginary eigenvalues. By
\eqref{eq:derivationsn28}, we know that
\begin{equation*}
f=\begin{pmatrix} A & 0 \\ B & \trC A \end{pmatrix}
\end{equation*}
for some $A\in \bC^{2\times 2}$ and $B\in \bR^{2\times 4}$. If we conjugate $f$ with an automorphism $F$ of $\mfn_{28}$ as in \eqref{eq:innerautosn28} with $C=I_2$, we surely get a Lie algebra $\mfn_{28}\rtimes_{F f F^{-1}} \bR$ which is isomorphic to $\mfg$, where
\begin{equation*}
F f F^{-1}=\begin{pmatrix} A & 0 \\ B+D(A-(\trC A)I_2) & \trC A \end{pmatrix}.
\end{equation*}
Thus, $FfF^{-1}$ is block-diagonal if $A-(\trC A)I_2$ is invertible, i.e.\ if $\trC A$ is not a complex eigenvalue of the complex matrix $A$. However, if $\trC A$ would be an eigenvalue of the complex matrix $A$, then the other complex eigenvalue would have to be zero and so the real matrix $f$ would have one eigenvalue equal to zero, which is excluded since $f$ has no purely imaginary eigenvalues.

Thus, calling $FfF^{-1}$ again $f$, we may assume that $f=\diag(A,\trC A)$. But then we may use an automorphism of $\mfn_{28}$ as in \eqref{eq:innerautosn28} to bring $A$ into complex Jordan normal form.
Hence, we may assume that either $A=\diag(w_1,w_2)$
for $w_1,w_2\in \bC$ with $\re(w_1)\neq 0$, $\re(w_2)\neq 0$, $\re(w_1+w_2)=\trC A\neq 0$ or
\begin{equation*}
A=\begin{pmatrix} w & 1 \\ 0 & w \end{pmatrix}
\end{equation*}
for some $w\in \bC$ with $\re(w)\neq 0$. We provide now in both cases an example of an $\SU(3)$-structure $(\omega,\rho)\in \Lambda^2 \mfn_{28}^*\times \Lambda^3 \mfn_{28}^*$
a two-form $\nu\in \Lambda^2\mfn_{28}^*$ and a one-form $\alpha\in \mfn_{28}^*$ such that
$\rho=d\nu$ and such that \eqref{eq:3} is fullfilled, where the latter equation is valid by Lemma \ref{le:implicationsofeq3} if and only if
$f.\nu^{1,1}=\omega$, $f.\nu^{2,0}=d\alpha$. We will always choose $\alpha=0$ and a $(1,1)$-form $\nu$, so that the second equation is automatically fulfilled and we only have to deal with the first one.

In the first case, one checks by a straightforward computation that
\begin{equation*}
\omega=\lambda^2 \omega_0,\quad \rho=\lambda^3 \rho_0,\quad
\nu=\frac{\lambda^2}{2\, \re(w_1)} e^{12}+\frac{\lambda^2}{2\, \re(w_2)} e^{34}+\lambda^3 e^{56}
\end{equation*}
with $\lambda:=-\frac{1}{2\,\re(w_1+w_2)}$ fulfills all necessary equations, whereas in the second case
\begin{equation*}
\omega=\lambda^2 \omega_0,\quad \rho=\lambda^3 \rho_0,\quad
\nu=  -2 \lambda^3 e^{12}-(16 \lambda^5+2\lambda^3) e^{34}-4\lambda^4\cdot (e^{14}-e^{23})+\lambda^3 e^{56}
\end{equation*}
with $\lambda:=-\frac{1}{4 \re(w)}$ does the job.

The second statement in the assertion follows immediately from the considerations above by noting that rescaling $f$ by a non-zero scalar gives an isomorphic Lie algebra and by noting that we may order the real parts of the eigenvalues of $A$ in such a way that the first one is greater or equal to the second one.
\end{proof}
\begin{remark}
Note that $f_{a,b_1,b_2}$ and $h_b$ in Theorem \ref{th:exactG2n28} both fix $e^{56}$ and so one easily sees that
$(\omega,\rho)$ as in Lemma \ref{le:canonicalformsonn28} with $\lambda=1$, i.e.\ $\omega=-e^{12}-e^{34}+e^{56}$, $\rho=e^{136}-e^{246}-e^{145}-e^{235}$, give rise to an exact $\G_2$-structure on $\mfn_{28} \rtimes_{f_{a,b_1,b_2}}\bR$ and $\mfn_{28} \rtimes_{h_b}\bR$, respectively. This explains the strange `normalisation' of the endomorphisms in Theorem \ref{th:exactG2n28}.
\end{remark}
Looking for exact $\G_2$-structures of special torsion, we obtain:
\begin{theorem}\label{th:exactG2withspecialtorsion}
	Let $\mfg$ be a seven-dimensional almost nilpotent Lie algebra with codimension-one nilpotent ideal
	isomorphic to $\mfn_{28}$ which admits an exact $\G_2$-structure. Then:
	
	\begin{itemize}
		\item[(a)]
		$\mfg$ admits an exact $\G_2$-structure with special torsion of negative type.
		\item[(b)]
		$\mfg$ admits an exact $\G_2$-structure with special torsion of positive type
		if and only if $\mfg\not\cong \mfn_{28}\rtimes_{f_{-\qart,b,b}}\bR$ for all $b\in \bR$.
	\end{itemize}
\end{theorem}
\begin{proof}
By Theorem \ref{th:exactG2n28}, we may assume that $\mfg=\mfn_{28}\rtimes_f \bR$ with
either $f=f_{a,b_1,b_2}$ for certain $a\in \left[-\tfrac14,\infty\right)$, $b_1,b_2\in \bR$
or $f=h_b$ for some $b\in \bR$.
We divide the proof into four different parts:
\begin{itemize}
	\item[(I)]
We first show that $\mfg=\mfn_{28}\rtimes_{f_{a,b_1,b_2}} \bR$ for $a,b_1,b_2\in \bigl(\left[-\tfrac{1}{4},\tfrac{1}{2}\right]\setminus \{0\}\bigr)\times \bR^2$ with $a\neq -\tfrac{1}{4}$ or
$b_1\neq b_2$ and $\mfn_{28}\rtimes_{h_b} \bR$ for $b\in \bR$ admits an exact $\G_2$-structure
with special torsion of positive type.

For this, note that under the assumptions on $a,b_1,b_2$, the Lie algebra $\mfn_{28}\rtimes_{f_{a,b_1,b_2}}\bR$ is isomorphic to $\mfn_{28}\rtimes_{g} \bR$ with 
		\begin{equation*}
		g=g_{a,b_1,b_2,c}=\left(\begin{smallmatrix} a+i b_1 & c &  \\  & -\tfrac{1}{2}-a+i b_2 &  \\ &  & -\tfrac{1}{2} + i(b_1+b_2) \end{smallmatrix}\right)
	    \end{equation*}
	   If $a= -\tfrac{1}{4}$ and $b_1=b_2=:b$, then, for any $c\in \bR\setminus \{0\}$, we have $\mfn_{28}\rtimes_{g_{-\qart,b,b,c}} \bR \cong \mfn_{28}\rtimes_{h_b} \bR$.
   
   So we are looking for exact $\G_2$-structures with special torsion of positive type on $\mfn_{28}\rtimes_{g} \bR$. For this, note that $a\neq 0$ and $2a+1\neq 0$ by assumption. Thus,
  \begin{equation*}
  \begin{split}
  \nu&=\frac{1}{2a} e^{12}+\frac{2c^2-4a(b_1-b_2)^2+1)}{a(4(b_1-b_2)^2+1)(2a+1)} e^{34}+e^{56}-\frac{2(b_1-b_2)c}{a(4(b_1-b_2)^2+1)} (e^{13}+e^{24})\\
  &+\frac{c}{a(4(b_1-b_2)^2+1)} (e^{14}-e^{23})\in \Lambda^2 \mfn_{28}^*
  \end{split}
  \end{equation*}
  is well-defined and one checks that $d\nu=\rho_0=e^{136}-e^{246}-e^{145}-e^{235}$
  and $g.\nu=\omega_0=-e^{12}-e^{34}+e^{56}$. Thus, the pair $(\omega,\rho)$
  gives rise to an exact $\G_2$-structure $\varphi$. As $\hat{\rho}=e^{135}-e^{146}-e^{236}-e^{245}$,
  we get
   \begin{equation*}
   \begin{split}
   d\star_{\varphi}\varphi&=d\left(\tfrac{1}{2}\omega^2+e^7\wedge \hat{\rho}\right)=
e^7\wedge g.\left(\tfrac{1}{2}\omega^2\right)-e^7\wedge d_{\mfn_{28}}\hat{\rho}\\
&=e^7\wedge\left(-3 e^{1234}+ (2a-1)e^{1256}-(2+2a)e^{3456}+c (e^{1456}-e^{2356})\right)
   \end{split}
   \end{equation*}
due to $d(\omega^2)=0$. Hence, the torsion two-form $\tau$ is given by
\begin{equation*}
\tau=-\star_{\varphi}d\star_{\varphi} \varphi=(2+2 a) e^{12}-(2a-1)e^{34}+c (e^{14}-e^{23})+3 e^{56}.
\end{equation*}
Now the exact $\G_2$-structure $\varphi$ has special torsion of positive type if and only if $\tau^3=0$,
which is equivalent to $((2+2 a) e^{12}-(2a-1)e^{34}+c (e^{14}-e^{23}))^2=0$, and so to
\begin{equation*}
0=(2+2a)(2a-1)+c^2=4a^2+2a-2+c^2.
\end{equation*}
Here, $a$ is fixed and we are searching for a solution of this equation for $c$, which is possible if $4a^2+2a-2\leq 0$, i.e.\ if $a\in \left[-\tfrac{1}{4},\tfrac{1}{2}\right]$. Note that for $a=-\tfrac{1}{4}$, we have $c=\pm \tfrac{3}{2}\neq 0$ and so $\mfn_{28}\rtimes_{h_b}\bR$ admits an exact $\G_2$-structure with special torsion of positive type for any $b\in \bR$.
\item[(II)]
We show now that $(\omega_0,\rho_0)$ defines also an exact $\G_2$-structure on $\mfn_{28}\rtimes_{g_{-\qart,b,b,c}} \bR \cong \mfn_{28}\rtimes_{h_b} \bR$ with special torsion of negative type for a suitable chosen $c\in \bR\setminus \{0\}$.

The computations in (I) show that $\tau=\tfrac{3}{2} (e^{12}+e^{34})+c (e^{14}-e^{23})+3 e^{56}$. Hence, $\varphi$ has special torsion of negative type if and only if $\tfrac{2}{3}\left|\tau\right|_{\varphi}^6=\left|\tau^3\right|_{\varphi}^2$, which here is equivalent to
\begin{equation*}
\tfrac{2}{3} (\tfrac{27}{2}+2c^2)^3=\left(18\left(\tfrac{9}{4}-c^2\right)\right)^2 \quad \Longleftrightarrow\quad c^2\left(\tfrac{16}{3}c^4-216c^2+2187\right)=0,
\end{equation*}
i.e.\ to $c=0$ or $c=\pm \tfrac{9}{2}$. Thus, for $c=\tfrac{9}{2}$, we get an exact $\G_2$-structure with special torsion of negative type on $\mfn_{28}\rtimes_{g_{-\qart,b,b,\tfrac{9}{2}}} \bR\cong \mfn_{28}\rtimes_{h_b}\bR$.
\item[(III)]
Next, we show that $\mfn_{28}\rtimes_{f_{a,b_1,b_2}}\bR$ admits an exact $\G_2$-structure with special torsion of positive type if $(a,b_1,b_2)\in \left(\tfrac{1}{2},\infty\right)\times \bR^2$ and that it admits an exact $\G_2$-structure with special torsion of negative type for any possible values of $(a,b_1,b_2)$, i.e.\ for any $(a,b_1,b_2)\in \Bigl(\left[\tfrac{1}{4},\infty\right)\setminus \{0\}\Bigr)\times \bR^2$.

For this, we note that $\mfn_{28}\rtimes_{f_{a,b_1,b_2}}\bR$ is isomorphic to $\mfn_{28}\rtimes_h\bR$ with
	\begin{equation*}
h:=h_{a,b_1,b_2,r}:=\left(\begin{smallmatrix} a & -b_1 & & & & \\
                                  b_1 & a & & & &  \\
                                  & & -\tfrac{1}{2}-a & -b_2 & & \\
                                   & & b_2 & -\tfrac{1}{2}-a & & \\
                                    & & r & & -\tfrac{1}{2} & -(b_1+b_2) \\
                                    & & & -r & b_1+b_2 & -\tfrac{1}{2}
                              \end{smallmatrix}\right)
\end{equation*}
for any $r\in \bR$. The minus sign occuring before one of the $r$s is due to $(e_1,-e_2,e_3,-e_4,e_6,-e_5)$ being a complex basis, i.e.\ due to the shift in the order of $e_5$ and $e_6$.

We have $d\nu=\rho_0$ and $h.\nu=\omega_0$ for
\begin{equation*}
\begin{split}
\nu&=\tfrac{1}{2 a}e^{12}-\tfrac{(b_1+2b_2)^2+(2 r^2+a+1)(a+1)}{(2a+1)((a+1)^2+(b_1+2b_2)^2)}e^{34}+e^{56}
+\tfrac{r (b_1+2b_2)}{(a+1)^2+(b_1+2b_2)^2}\left(e^{35}-e^{46}\right)\\
&+\tfrac{r(1+a)}{(a+1)^2+(b_1+2b_2)^2}\left(e^{36}+e^{45}\right).
\end{split}
\end{equation*}
Hence, the pair $(\omega_0,\rho_0)$ defines an exact $\G_2$-structure $\varphi$ on $\mfn_{28}\rtimes_h\bR$ for any value of $r\in \bR$. Moreover, we have
\begin{equation*}
\begin{split}
d\star_{\varphi}\varphi&=e^7\wedge\left(h.\left(\tfrac{1}{2}\omega^2\right)-d_{\mfn_{28}}\hat{\rho}\right)\\
&=e^7\wedge\left(-3 e^{1234}+ (2a-1)e^{1256}-(2+2a)e^{3456}+r (e^{1236}+e^{1245})\right)
\end{split}
\end{equation*}
and so the torsion two-form $\tau$ is given by
\begin{equation*}
\tau=-\star_{\varphi}d\star_{\varphi} \varphi=(2+2 a) e^{12}-(2a-1)e^{34}-r (e^{36}+e^{45})+3 e^{56}.
\end{equation*}
Hence,
\begin{equation*}
\tau^3=-12(1+a) (6a-3-r^2) e^{123456}
\end{equation*}
and $\tau^3=0$, i.e.\ $\varphi$ has special torsion of positive type, if and only if $r^2=6a-3$. But we assumed $a>\tfrac{1}{2}$ and so have $6a-3>0$ and, consequently, $\varphi$ has special torsion of positive type for $r=\sqrt{6a-3}\in \bR$.

Moreover, $\varphi$ has special torsion of negative type if and only if
\begin{equation*}
\tfrac{2}{3}((2+2a)^2+(2a-1)^2+2r^2+9)^3=\left|\tau\right|_{\varphi}^6=\left|\tau^3\right|_{\varphi}^2=\bigl(12(1+a) (6a-3-r^2 )\bigr)^2.
\end{equation*}
Bringing both terms on one side and factorising gives
\begin{equation*}
\tfrac{16}{3}\left((a-2)^2+r^2\right)\cdot \left(8a^2+22a+5-r^2\right)^2=0.
\end{equation*}
So one may find some $r\in \bR$ such that $\varphi$ has special torsion of negative type if $(2a+5)(4a+1)=8a^2+22a+5\geq 0$. But this is the case if $a\geq -\tfrac{1}{4}$ and so $\mfn_{28}\rtimes_{f_{a,b_1,b_2}}\bR$ admits an exact $\G_2$-structure with special torsion of negative type for any possible values of $(a,b_1,b_2)$.
\item[(IV)]
Finally, we need to show that for any $b\in \bR$, the Lie algebra $\mfn_{28}\rtimes_{f_{-\qart,b,b}}\bR$ does not admit an exact 
$\G_2$-structure with special torsion of positive type.

For this, let $(\omega,\rho)$ be a half-flat $\SU(3)$-structure which determines an exact $\G_2$-structure $\varphi$
on $\mfn_{28}\rtimes_{f_{-\qart,b,b}}\bR$ and let $\nu\in \Lambda^2 \mfn_{28}^*$ and $\alpha\in \mfn_{28}$ be such that
\eqref{eq:3} holds. By Lemma \ref{le:canonicalformsonn28}, we may assume that
\begin{equation*}
\begin{split}
\omega&=\epsilon\lambda^2\omega_0=\epsilon\lambda^2 (-e^{12}-e^{34}+e^{56}),\quad \rho=\lambda^3 \rho_0=\lambda^3(e^{136}-e^{246}-e^{145}-e^{235})
\end{split}
\end{equation*}
for some $\lambda\in \bR\setminus \{0\}$, up to an automorphism $F$ of $\mfn_{28}$, i.e.\ $(\omega,\rho)$ are of this 
form on $\mfn_{28}\rtimes_{F f_{-\qart,b,b}F^{-1}}\bR$. Now one computes that $f:=F f_{-\qart,b,b}F^{-1}$ is of the form
\begin{equation*}
f=\begin{pmatrix} \left(-\frac14+ib\right) I_2 &  \\ B & -\tfrac{1}{2}+2i b \end{pmatrix}
\end{equation*}
for some $B\in \bR^{2\times 4}$. By Lemma \ref{le:implicationsofeq3}, we have $[f,J_0]=0$, which amounts to $B$ being of the form
\begin{equation*}
B=\begin{pmatrix} a_1 & a_2 & a_3 & a_4 \\ a_2 & -a_1 & a_4 & -a_3 \end{pmatrix}
\end{equation*}
for certain $a_1,a_2,a_3,a_4\in \bR$. Moreover, by Lemma \ref{le:implicationsofeq3}, we have $f.\nu^{1,1}=\omega$ and
\begin{equation*}
\nu^{1,1}=\lambda^3 e^{56}+e^5\wedge \beta+e^6\wedge J^*\beta+c_1 e^{12}+c_2 e^{34} 
\end{equation*}
for certain $\beta\in \spa{e^1,e^2,e^3,e^4}$ and $c_1,c_2\in \bR$ by Lemma \ref{le:canonicalformsonn28}. Thus,
\begin{equation*}
\epsilon\lambda^2 (-e^{12}-e^{34}+e^{56})=\omega=f.\nu^{1,1}=\lambda^3 e^{56}+e^5\wedge \gamma+e^6\wedge J^*\gamma+ \tfrac{c_1}{2}e^{12}+\tfrac{c_2}{2} e^{34}
\end{equation*}
for some $\gamma\in \spa{e^1,e^2,e^3,e^4}$, which implies, in particular, $\lambda=\epsilon$, i.e.\ $\omega=\epsilon\omega_0$ and $\rho=\epsilon\rho_0$. Since for $\epsilon=-1$, the induced orientation is the opposite of that for $\epsilon=1$, we always have $\hat{\rho}=e^{135}-e^{146}-e^{236}-e^{245}$. Thus, one cpmputes
\begin{equation*}
\end{equation*}
\begin{equation*}
\begin{split}
d\star_{\varphi}\varphi=& e^7\wedge \left(f.\left(\tfrac{1}{2}\omega^2\right)-d_{\mfn_{28}}\hat{\rho}\right)\\
=& e^7\wedge\left(-3 e^{1234}-\tfrac{3}{2}e^{1256}-\tfrac{3}{2}e^{3456}+a_1 \left(e^{2345}+ e^{1346}\right)+a_2 \left(e^{2346}- e^{1345}\right)\right.\\[-3pt]
&\left.+a_3 \left(e^{1245}+ e^{1236}\right)+a_4 \left(e^{1246}- e^{1235}\right)\right)
\end{split}
\end{equation*}
independently of $\epsilon$. Hence,
\begin{equation*}
\begin{split}
\tau=&-\star_{\varphi}d\star_{\varphi} \varphi\\
&=\epsilon\left(\tfrac{3}{2} \left(e^{12}+e^{34}\right)+ 3 e^{56}-a_1 \left(e^{16}+e^{25}\right)+a_2 \left(e^{15}-e^{26}\right)
-a_3 \left(e^{36}+e^{45}\right)+a_4 \left(e^{35}-e^{46}\right)\right),
\end{split}
\end{equation*}
and so
\begin{equation*}
\tau^3=9\epsilon \left(\tfrac{9}{2}+a_1^2+a_2^2+a_3^2+a_4^2\right)e^{123456}\neq 0,
\end{equation*}
i.e.\ $\varphi$ does not have special torsion of positive type.
\end{itemize}
\end{proof}
Finally in this section, we show that a Lie algebra of the form $\mfg=\mfn_{28}\rtimes_f \bR$ cannot admit a closed $\G_2$-eigenform:
\begin{theorem}\label{th:noeigenformsn28}
Let $\mfg$ be a seven-dimensional almost nilpotent Lie algebra with codimension-one nilpotent ideal isomorphic to $\mfn_{28}$. Then
$\mfg$ does not admit a closed $\G_2$-eigenform.
\end{theorem}
\begin{proof}
We assume the contrary. Then, by Lemma \ref{le:canonicalformsonn28}, $\lambda\in \bR\setminus \{0\}$, $\epsilon\in \{-1,1\}$ such that $\omega=\epsilon \lambda \omega_0$, $\rho=\lambda^3\rho_0$
and $\nu \in [\Lambda^{1,1}_0 \mfn_{28}^*]$ is as in \eqref{eq:canonicalformsonn28two}, i.e.
\begin{equation*}
\nu=c e^{12}+(\lambda^3-c) e^{34}+\lambda^3  e^{56}+e^5\wedge \beta+e^6\wedge J_0^*\beta
\end{equation*}
for certain $c\in \bR$ and $\beta\in V^*$. Moreover, we may assume that
\begin{equation*}
f.\nu= \omega,\qquad \omega\wedge (f.\omega-\nu)=d\hat{\rho}.
\end{equation*}
for some $f\in \mathrm{Der}(\mfn_{28})$. Since then also $(-\omega,\rho,\nu,-f)$ fulfills all necessary equations, and so defines a closed $\G_2$-eigenform, we may assume that $\epsilon=1$.

But then one computes
\begin{equation*}
\begin{split}
d\hat{\rho}&=\lambda^3 d\hat{\rho}_0=4 \lambda^3\, e^{1234}=\lambda^2 (-e^{12}-e^{34}+e^{56})\wedge -2\lambda (e^{12}+e^{34}+e^{56})\\
&=\omega\wedge -2\lambda (e^{12}+e^{34}+e^{56}).
\end{split}
\end{equation*}
Since wedging with $\omega$ is an isomorphism from $\Lambda^2\mfh^*$ to $\Lambda^4 \mfh^*$, the latte equation implies
\begin{equation*}
f.\omega-\nu=-2\lambda (e^{12}+e^{34}+e^{56})
\end{equation*}
By \eqref{eq:derivationsn28}, we know that
\begin{equation*}
f=\begin{pmatrix} A & 0 \\ B & \trC A \end{pmatrix}
\end{equation*}
for some $A=(a_{ij})_{i,j}\in \bC^{2\times 2}$ and some $B\in \bR^{2\times 4}$. Thus, inserting $(e_1,e_2)$ into the equality $f.\omega-\nu=-2\lambda (e^{12}+e^{34}+e^{56})$ yields
\begin{equation*}
2\,\re(a_{11})\lambda^2-c=-2\,\re(a_{11})\omega(e_1,e_2)-\nu(e_1,e_2)=(f.\omega-\nu)(e_1,e_2)=-2\lambda.
\end{equation*}
Similarly, we obtain
\begin{equation*}
2\,\re(a_{22})\lambda^2-(\lambda^3-c)=(f.\omega-\nu)(e_3,e_4)=-2\lambda.
\end{equation*}
by inserting $(e_3,e_4)$. Adding these two equations yields
\begin{equation}\label{eq:auxiliary}
2\,\re(\trC A)\lambda^2-\lambda^3=-4\lambda.
\end{equation}
Moreover, inserting $(e_5,e_6)$, we do get
\begin{equation*}
-2\,\re(\trC A)\lambda^2-\lambda^3=(f.\omega-\nu)(e_5,e_6)=-2\lambda.
\end{equation*}
Adding \eqref{eq:auxiliary} to this equation, we obtain
$-2\lambda^3=-6\lambda$, and so, since $\lambda\neq 0$, that $\lambda^2=3$.

However, we also get
\begin{equation*}
-2\,\re(\trC A)\lambda^3=-2\re(\trC A)\nu(e_5,e_6)=f.\nu(e_5,e_6)=\omega(e_5,e_6)=\lambda^2,
\end{equation*}
i.e.\ $2\re(\trC A)\lambda^2=\lambda$, which, together with \eqref{eq:auxiliary}, yields
$\lambda-\lambda^3=-4\lambda$, i.e.\ $\lambda^2=5$, a contradiction.

Thus, $\mfg$ does not admit a closed $\G_2$-eigenform.
\end{proof}
\medbreak\noindent \textsc{Acknowledgements.} The first author was
supported by a \emph{Forschungs\-stipendium} (FR 3473/2-1) from
the Deutsche Forschungsgemeinschaft (DFG).

\end{document}